\documentclass[a4paper,11pt]{amsart}
\usepackage{amsmath,amsthm,amssymb}
\usepackage[latin1]{inputenc}
\usepackage[T1]{fontenc}
\usepackage[all]{xy}
\usepackage{mathrsfs}
\usepackage[colorlinks=true,linkcolor=blue,pagebackref,breaklinks]{hyperref}

\usepackage{dsfont}
\usepackage{mathabx}
\usepackage[hmargin={1.25in,1.25in},vmargin={1.5in,1.5in}]{geometry}

\newtheorem{thm}{Theorem}[section]
\newtheorem{prop}{Proposition}[section]
\newtheorem{df}{Definition}[section]
\newtheorem{lem}{Lemma}[section]
\newtheorem{cor}{Corollary}[section]
\newtheorem{dflem}{Definition-Lemma}[section]

\newtheorem{ex}{Example}[section]
\newtheorem{conj}{Conjecture}[section]
\newtheorem{rem}{Remark}[section]

\newenvironment{dem}{\paragraph{Proof}}
{\begin{flushright}$\Box$\end{flushright}}

\newcommand{\Z}{\mathbb{Z}}
\newcommand{\Q}{\mathbb{Q}}
\newcommand{\C}{\mathbb{C}}
\newcommand{\Ss}{\mathbb{S}}
\newcommand{\R}{\mathbb{R}}

\newcommand{\F}{\mathbb{F}}

\newcommand{\AK}{\mathbb{A}_{K}}
\newcommand{\AF}{\mathbb{A}_{F}}

\newcommand{\AFf}{\mathbb{A}_{F,f}}

\newcommand{\AFp}{\mathbb{A}_{F^{+}}}

\newcommand{\AQ}{\mathbb{A}_{\mathbb{Q}}}
\newcommand{\AQf}{\mathbb{A}_{\mathbb{Q},f}}
\newcommand{\lieG}{\mathfrak{g}}
\newcommand{\lieK}{\mathfrak{k}}
\newcommand{\lieP}{\mathfrak{p}}
\newcommand{\lieH}{\mathfrak{h}}

\def\<{\langle}
\def\>{\rangle}

\def\R{\mathbb R}

\def\Z{\mathbb Z}

\def\Q{\mathbb Q}
\def\C{\mathbb C}
\def\F{\mathbb F}

\def\cm{F}

\def\Acm{\mathbb{A}_{F}}

\def\Atr{\mathbb{A}_{F^{+}}}

\makeatletter 
\@addtoreset{equation}{section}
\makeatother  

\title{Factorization of arithmetic automorphic periods}
\author{Jie LIN}

\begin{document}
\thanks{ J.L. was supported by the European Research Council under the European Community's Seventh Framework Programme (FP7/2007-2013) / ERC Grant agreement no. 290766 (AAMOT)}
\maketitle

\begin{abstract}
In this paper, we prove that the arithmetic automorphic periods for $GL_{n}$ over a CM field factorize through the infinite places. This generalizes a conjecture of Shimura in 1983, and is predicted by the Langlands correspondence between automorphic representations and motives.
\end{abstract}

\tableofcontents

\section*{Introduction}

\text{}
The aim of this paper is to prove the factorization of arithmetic automorphic periods defined as Petersson inner products of arithmetic automorphic forms on unitary Shimura varieties. This generalizes a conjecture of Shimura (c.f. Conjecture $5.12$ of \cite{shimura83}).

We first introduce the conjecture of Shimura to illustrate our main result. Let $E$ be a totally real field of degree $d$. Let $J_{E}$ be the set of real embeddings of $E$. 
Let $f$ be an arithmetic Hilbert cusp form inside a cuspidal automorphic representation $\pi$ of $GL_{2}(\mathbb{A}_{E})$. We define the period $P(\pi)$ as the Petterson inner product of $f$. One can show that up to multiplication by an algebraic number, the period $P(\pi)$ does not depend on the choice of $f$.

For each $\sigma\in J_{E}$, Shimura conjectured the existence of a complex number $P(\pi,\sigma)$, associated to a quaternion algebra which is split at $\sigma$ and ramified at other infinite places, such that:
\begin{equation}\label{Shimura 1}
P(\pi)\sim \prod\limits_{\sigma\in J_{E}}P(\pi,\sigma)
\end{equation}
where the relation $\sim$ means equality up to multiplication by an algebraic number.

Furthermore, if $D$ is a quaternion algebra and $\pi$ admets a Jacquet-Langlands transfer $\pi_{D}$ to $D$, we may define $P(\pi_{D})$ as Petersson inner product of an algebraic form in $\pi_{D}$, and Shimura conjectured that:
\begin{equation}\label{Shimura 2}
P(\pi_{D})\sim \prod\limits_{\sigma\in J_{E},\text{ split for }D}P(\pi,\sigma).\end{equation}

This conjecture was proved under some local hypotheses in an important paper of M. Harris (c.f. \cite{harrisunitary}) and was improved by H. Yoshida (c.f. \cite{yoshidashimuraconjecture}). The paper of M. Harris is very long and involves many techniques which seems extremely difficult to generalize. In this paper, we prove a generalization of Shimura's conjecture (c.f. Conjecture \ref{factorization conjecture}) by a new and simpler method.

We consider representations of $GL_{n}(\Acm)$ where $F$ is a CM field. We write $J_{F}$ for the set of complex embeddings of $F$. We fix $\Sigma$ a CM type of $F$, i.e., $\Sigma\subset J_{F}$ presents $J_{F}$ modulo the action of complex conjugation.

Let $F^{+}$ be the maximal totally real subfield of $F$. Instead of the quaternions algebras, we consider unitary groups of rank $n$ with respect to $F/F^{+}$. They are all inner forms of $GL_{n}(\AF^{+})$.  

We use $I$ to denote the signature of a unitary group. It can be considered as a map from $\Sigma$ to $\{0,1,\cdots,n\}$. For each $I$, let $U_{I}$ be a unitary group of signature $I$. We note that $U_{I,F}\cong GL_{n,F}$ as algebraic group over $F$. In particular, we have $U_{I}(\Acm)\cong GL_{n}(\Acm)$. We assume that $\Pi$, considered as a representation of $U_{I}(\Acm)$, descends by base change to $U_{I}(\Atr)$. We refer to \cite{arthurtraceformula}, \cite{harrislabesse}, \cite{labesse} or \cite{endoscopicfour} for details of base change.

We can then define a period $P^{(I)}(\Pi)$ as Petersson inner product of an algebraic automorphic form in the bottom degree of cohomology of the similitude unitary Shimura variety attached to $U_{I}$. The construction is given in section $2$.

 \begin{conj}
 There exists some non zero complex numbers $P^{(s)}(\Pi,\sigma)$ for all $0\leq s\leq n$ and $\sigma\in\Sigma$ such that 
 \begin{equation}\label{main conjecture intro}
 P^{(I)}(\Pi) \sim_{E(\Pi)} \prod\limits_{\sigma\in\Sigma}P^{(I(\sigma))}(\Pi,\sigma)
 \end{equation} for any $I=(I(\sigma))_{\sigma\in\Sigma}\in \{0,1,\cdots,n\}^{\Sigma}$.
  \end{conj}  
  
  Our main theorem is the following (c.f. Theorem \ref{complete theorem}):
  
  \begin{thm}
The above conjecture is true provided that $\Pi$ is $2$-regular with a global non vanishing condition which is automatically satisfied if $\Pi$ is $6$-regular.
\end{thm}

We remark that this conjecture is not as simple as it may look like even if we do not put any restriction on the complex numbers $P^{(s)}(\Pi,\sigma)$. In fact, the number of different $P^{(I)}(\Pi)$ is $d^{n+1}$ and the number of different $P^{(I(\sigma))}(\Pi,\sigma)$ is only $d(n+1)$. 
 
On the other hand, it is true that the choice of $P^{(I(\sigma))}(\Pi,\sigma)$ is not unique. We have specified a canonical choice in section $3.5$. Similarly to Shimura's formulation, the canonical choice of $P^{(I(\sigma))}(\Pi,\sigma)$ is related to the unitary group of signature $(1,n-1)$ at $\sigma$ and $(0,n)$ at other places (c.f. section $4.4$ of \cite{harrislin}). The author proved that the periods $P^{(I)}(\Pi)$ as well as the local specified periods $P^{(I)}(\Pi,\sigma)$ are functorial in the sense of Langlands functoriality in \cite{lincomptesrendus} and \cite{linthesis}, .

We also get a partial result with a weaker regular condition (c.f. Theorem \ref{restricted theorem}):
  
 \begin{thm}
If $n\geq 4$ and $\Pi$ satisfies a global non vanishing condition, in particular, if $\Pi$ is $3$-regular, then there exists some non zero complex numbers $P^{(s)}(\Pi,\sigma)$ for all $1\leq s\leq n-1$, $\sigma\in\Sigma$ such that $P^{(I)}(\Pi) \sim_{E(\Pi)} \prod\limits_{\sigma\in\Sigma}P^{(I(\sigma))}(\Pi,\sigma)$ for all $I=(I(\sigma))_{\sigma\in\Sigma}\in \{1,2,\cdots,n-1\}^{\Sigma}$.
 \end{thm}  
 
Before introducing the proof, let us look at equation (\ref{main conjecture intro}) that we want to prove. The left hand side involves $d^{n+1}$ periods and the right hand side involves only $d(n+1)$ periods. This is only possible if there are many relations between the periods $P^{(I)}(\Pi)$ in the left hand side.
 
In fact, the first step of the our proof is to reduce Conjecture $0.1$ to relations of the periods. The general argument is given in section $3.1$.

The next step is then to prove these relations. The proof involves several techniques like CM periods, Whittaker periods and special values of $L$-functions. We use three results on special values of $L$-functions. The first one is due to Blasius on relations between special values of $L$-functions for Hecke characters and CM periods (c.f. section $1$ of \cite{harrisunitary}). The second one relates special values of $L$-functions for $GL_{n}*GL_{1}$ and the arithmetic automorphic periods $P^{(I)}(\Pi)$ (c.f. \cite{guerberofflin}). The last one is about relations between special values of $L$-functions for $GL_{n}*GL_{n-1}$ and the Whittaker periods which is proved in \cite{harrismotivic} over quadratic imaginary field, and in \cite{haraldappendix} for general CM fields.

The advantage of the last result is that the $GL_{n-1}$-representation do not need to be cuspidal. This allows us to construct auxiliary representations of $GL_{n-1}$ more freely and leads to relations between Whittaker periods and arithmetic automorphic periods $P^{(I)}(\Pi)$ (see Theorem $3.1$) which generalizes Theorem $6.7$ of \cite{harrismotivic}. This relation already implies the partial result mentioned above.

To prove the whole conjecture, a more ingenious construction needs to be made. We construct carefully a non-cuspidal representation of $GL_{n+2}(\AF)$ related to $\Pi$, and an auxiliary cuspidal representation of $GL_{n+3}(\AF)$. The $GL_{n+3}(\AF)$ representation is induced from Hecke characters. Hence special values of its $L$-function can be written in terms of CM periods by Blasius's result. The details can be found in section $3.4$.

The manipulation of different special values with different auxiliary representations can give many interesting results of period relations or special values of $L$-functions. We refer the reader to \cite{lincomptesrendus}, \cite{harrismotivic} or \cite{linthesis} for more examples. More recently, the author and H. Grobner proved some results on special values which implies one case of the Ichino-Ikeda conjecture up to multiplication by an algebraic number in a very general setting (c.f. \cite{grobnerlin}).

We remark at last that Conjecture $0.1$ is predicted by motivic calculation (c.f. section $2.3$ of \cite{harrislin}). More generally, the motivic calculation and the Langlands correspondence predict the existence of more automorphic periods and some finer relations between them. This will be discussed in details in a forthcoming paper of the author with H. Grobner and M. Harris.

\subsection*{Acknowledgement}
I would like to thank Michael Harris for introducing this interesting question and also for helpful discussions. This paper is part of my thesis. I am grateful to Henri Carayol, Kartik Prasanna and Harald Grobner for their careful reading and useful comments.

\section{Preliminaries}

\subsection{Basic Notation}
Let $F$ be a CM field and $F^{+}$ be its maximal totally real subfield. We denote by $J_{F}$ the set of embeddings of $F$ in $\C$. The complex conjugation $c$ acts on the set $J_{F}$. We say $\Sigma$ a subset  of $J_{F}$ is a CM type if $J_{F}$ is the disjoint union of $\Sigma$ and $\Sigma^{c}$. For $\iota\in J_{E}$, we also write $\bar{\iota}$ for the complex conjugation of $\iota$.

As usual, we let $S$ be a finite set of places of $F$, containing all infinite places and all ramified places of any representation which will appear in the text.

Let $\chi$ be a Hecke character of $\cm$. We write $\chi_{\iota}$ by $z^{a_{\iota}}\bar{z}^{a_{\bar{\iota}}}$ for $\iota\in J_{F}$. We say that $\chi$ is {\it algebraic} if $a_{\iota}, a_{\bar{\iota}}\in \Z$ for all $\iota\in J_{F}$. We say that $\chi$ is {\it critical} if it is algebraic and moreover $a_{\iota}\neq a_{\bar{\iota}}$ for all $\iota\in  J_{\cm}$. It is equivalent to that the motive associated to $\chi$ has critical points in the sense of Deligne (cf.\ \cite{deligne79}). We remark that $0$ and $1$ are always critical points in this case.

Moreover, we write $\widecheck{\chi}$ for the Hecke character $\chi^{c,-1}$. Apparently if $\chi$ is algebraic or critical then so is $\widecheck{\chi}$.

For $\Pi$ an algebraic automorphic representation of $GL_{n}(\Acm)$, we know that for each $\iota\in J_{\cm}$, there exists $a_{\iota,1},\cdots, a_{\iota,n}, a_{\bar{\iota},1},\cdots, a_{\bar{\iota},n} \in \Z+\frac{n-1}{2}$ such that $$\Pi_{\iota}\cong {\rm Ind}_{B(\C)}^{G(\C)}[z_1^{a_{\iota,1}}\bar z^{a_{\bar{\iota},1}}_1 \otimes ...\otimes z_n^{a_{\iota,n}}\bar z^{a_{\bar{\iota},n}}_n].$$
Here $Ind$ refers to the normalised parabolic induction. We define the \textit{infinity type} of $\Pi$ at $\iota$ by $\{z^{a_{\iota,i}}\bar{z}^{a_{\bar{\iota},i}}\}_{1\leq i\leq n}$ (c.f. section $3.3$ of \cite{clozel}).

Let $N$ be any positive integer. We say that $\Pi$ is \textit{$N$-regular} if $| a_{\iota,i}-a_{\iota,j} | \geq N$ for any $\iota\in J_F$ and $1\leq i<j\leq n$. We say $\Pi$ is \textit{regular} if it is $1$ regular.

Throughout the text, we fix $\Sigma$ any CM type of $F$. We also fix $\psi$ an algebraic Hecke character of $F$ with infinity type $z^{1}\overline{z}^{0}$ at each place in $\Sigma$ such that $\psi\psi^{c}=||\cdot||_{\AK}$ (see Lemma $4.1.4$ of \cite{CHT} for its existence). It is easy to see that the restriction of $||\cdot||_{\AK}^{\frac{1}{2}}\psi$ to $\AQ^{\times}$ is the quadratic character associated to the extension $K/\Q$ by the class field theory. Consequently our construction is compatible with that in \cite{harrismotivic} or \cite{grobnerlin}.

Let $E$ be a number field.  We consider it as a subfield of $\C$. Let $x$, $y$ be two complex numbers. We say $x\sim_{E} y$ if $y\neq 0$ and $x/y\in E$. This relation is symmetric but not transitive unless we know both numbers involved are non-zero.

The previous relation can be defined in an equivariant way for $Aut(\C)$-families. More precisely,  let $x=\{x(\sigma)\}_{\sigma\in Aut(\C)}$ and $y=\{y(\sigma)\}_{\sigma\in Aut(\C)}$ be families of complex numbers. We say $x\sim_{E} y$ and equivariant under the action of $Aut(\C/F)$ if either $y(\sigma)=0$ for all $\sigma$, either $y(\sigma)\neq 0$ with the following properties:
\begin{enumerate}
\item $x(\sigma)\sim_{\sigma(E)} y(\sigma)$ for all $\sigma\in Aut(\C)$;
\item $\tau\left(\cfrac{x(\sigma)}{y(\sigma)}\right)=\cfrac{x(\tau\sigma)}{y(\tau\sigma)}$  for all $\tau\in Aut(\C/F)$ and all $\sigma\in Aut(\C)$.
\end{enumerate}

Lemma $1.17$ of \cite{grobnerlin} says that if $E$ contains $F^{Gal}$ and $x(\sigma)$ and $y(\sigma)$ depends only on $\sigma\mid E$, then the second point above will imply the first point.

We remark that all the $L$-values and periods in this paper will be considered as $Aut(\C)$-families.

\subsection{Rational structures on certain automorphic representations} 
\text{}

Let $\Pi$ be an automorphic representation of $GL_{n}(\AF)$. 

We denote by $V$ the representation space for $\Pi_{f}$. For $\sigma\in Aut(\C)$, we define another $GL_{n}(\AFf)$-representation $\Pi_{f}^{\sigma}$ to be $V\otimes_{\C,\sigma}\C$. Let $\Q(\Pi)$ be the subfield of $\C$ fixed by $\{\sigma\in Aut(\C) \mid  \Pi_{f}^{\sigma} \cong \Pi_{f}\}$. We call it the \textbf{rationality field} of $\Pi$.

For $E$ a number field, $G$ a group and $V$ a $G$-representation over $\C$, we say $V$ has an \textbf{$E$-rational structure} if there exists an $E$-vector space $V_{E}$ endowed with an action of $G$ such that $V=V_{E}\otimes_{E}\C$ as representation of $G$. We call $V_{E}$ an $E$-rational structure of $V$. 

We denote by $\mathcal{A}lg(n)$ the set of algebraic automorphic representations of $GL_{n}(\AF)$ which are isobaric sums of cuspidal representations as in section $1$ of \cite{clozel}.

\begin{thm}{(Th\'eor\`eme $3.13$ in \cite{clozel})}

Let $\Pi$ be a regular representation in $\mathcal{A}lg(n)$. We have that:
\begin{enumerate} 
\item $\Q(\Pi)$ is a number field. 
\item $\Pi_{f}$ has a $\Q(\Pi)$-rational structure unique up to homotheties. 
\item For all $\sigma\in Aut(\C)$, $\Pi_{f}^{\sigma}$ is the finite part of a regular representation in $\mathcal{A}lg(n)$. It is unique up to isomorphism by the strong multiplicity one theorem. We denote it by $\Pi^{\sigma}$.
\end{enumerate}

\end{thm}

\begin{rem}
Let $n=n_{1}+n_{2}+\cdots+n_{k}$ be a partitian of positive integers and $\Pi_{i}$ be regular representations in $\mathcal{A}lg(n_{i})$ for $1\leq i\leq k$ respectively.

The above theorem implies that, for all $1\leq i\leq k$, the rational field $\Q(\Pi_{i})$ is a number field.

Let $\Pi=(\Pi_{1}||\cdot||_{\AK}^{\frac{1-n_{1}}{2}} \boxplus \Pi_{2}||\cdot||_{\AK}^{\frac{1-n_{2}}{2}}\boxplus\cdots \boxplus \Pi_{k}||\cdot||_{\AK}^{\frac{1-n_{k}}{2}})||\cdot||_{\AK}^{\frac{n-1}{2}}$ be the normalized isobaric sum of $\Pi_{i}$. It is still algebraic.

We can see from definition that $\Q(\Pi)$ is the compositum of $\Q(\Pi_{i})$ with $1\leq i\leq k$. Moreover, if $\Pi$ is regular, we know from the above theorem that $\Pi$ has a $\Q(\Pi)$-rational structure.
\end{rem}

\subsection{Rational structures on the Whittaker model}
\text{}
Let $\Pi$ be a regular representation in $\mathcal{A}lg(n)$ and then its rationality field $\Q(\Pi)$ is a number field.

We fix a nontrivial additive character $\phi$ of $\AF$. Since $\Pi$ is an isobaric sum of cuspidal representations, it is generic. Let $W(\Pi_{f})$ be the Whittaker model associated to $\Pi_{f}$ (with respect to $\phi_{f}$). It consists of certain functions on $GL_{n}(\AF\text{}_{,f})$ and is isomorphic to $\Pi_{f}$ as $GL_{n}(\AF\text{}_{,f})$-modules.

Similarly, we denote the Whittaker model of $\Pi$ (with respect to) $\phi$ by $W(\Pi)$.

\begin{df}{\textbf{Cyclotomic character}}

There exists a unique homomorphism $\xi: Aut(\C) \rightarrow \widehat{\Z}^{\times}$ such that for any $\sigma\in Aut(\C)$ and any root of unity $\zeta$, $\sigma(\zeta)=\zeta^{\xi(\sigma)}$, called the cyclotomic character.
\end{df}

For $\sigma\in Aut(\C)$, we define $t_{\sigma}\in (\widehat{\Z} \otimes_{\Z} \mathcal{O}_{F})^{\times}  =\widehat{\mathcal{O}_{F}}^{\times}$ to be the image of $\xi(\sigma)$ by the embedding $(\widehat{\Z})^{\times} \hookrightarrow (\widehat{\Z} \otimes_{\Z} \mathcal{O}_{F})^{\times} $. We define $t_{\sigma,n}$ to be the diagonal matrix $diag(t_{\sigma}^{-n+1},t_{\sigma}^{-n+2},\cdots,t_{\sigma}^{-1},1)\in GL_{n}(\AF\text{}_{,f})$ as in section $3.2$ of \cite{raghuramshahidi}. 

For $w\in W(\Pi_{f})$, we define a function $w^{\sigma}$ on $GL_{n}(\AF\text{}_{,f})$ by sending $g\in GL_{n}(\AF\text{}_{,f})$ to $\sigma(w(t_{\sigma,n}g))$. For classical cusp forms, this action is just the $Aut(\C)$-action on Fourier coefficients.

\begin{prop}(Lemma $3.2$ of \cite{raghuramshahidi} or Proposition $2.7$ of \cite{harrismotivic})

The map $w\mapsto w^{\sigma}$ gives a $\sigma$-linear $GL_{n}(\AF\text{}_{,f})$-equivariant isomorphism from $W(\Pi_{f})$ to $W(\Pi^{\sigma}_{f})$.

For any extension $E$ of $\Q(\Pi_{f})$, we can define an $E$-rational structure on $W(\Pi_{f})$ by taking the $Aut(\C/E)$-invariants.

Moreover, the $E$-rational structure is unique up to homotheties.

\end{prop}

\begin{dem}
The first part is well-known (see the references in \cite{raghuramshahidi}). mahnkopf05

For the second part, the original proof in \cite{raghuramshahidi} works for cuspidal representations. The key point is to find a nonzero global invariant vector. It is equivalent to finding a nonzero local invariant vector for every finite place. Then Theorem $5.1(ii)$ of \cite{jac-ps-shalika-mathann} is involved as in \cite{harrismotivic}.

The last part follows from the one-dimensional property of the invariant vector which is the second part of Theorem $5.1(ii)$ of \cite{jac-ps-shalika-mathann}.
\end{dem}

\subsection{Rational structures on cohomology spaces and comparison of rational structures}\label{cohomology}
\text{}
Let $\Pi$ be a regular representation in $\mathcal{A}lg(n)$. The Lie algebra cohomology of $\Pi$ has a rational structure. It is described in section $3.3$ of \cite{raghuramshahidi}. We give a brief summary here.

Let $Z$ be the center of $GL_{n}$. Let $\lieG_{\infty}$ be the Lie algebra of $GL_{n}(\R\otimes_{\Q}F)$. Let $S_{real}$ be the set of real places of $F$, $S_{complex}$ be the set of complex places of $F$ and $S_{\infty}=S_{real}\cup S_{complex}$ be the set of infinite places of $F$.

For $v\in S_{real}$, we define $K_{v}:=Z(\R)O_{n}(\R)\subset GL_{n}(F_{v})$. For $v\in S_{complex}$, we define $K_{v}:=Z(\C)U_{n}(\C)\subset GL_{n}(F_{v})$. We denote by $K_{\infty}$ the product of $K_{v}$ with $v\in S_{\infty}$, and by $K_{\infty}^{0}$ the topological connected component of $K_{\infty}$.

We fix $T$ the maximal torus of $GL_{n}$ consisting of diagonal matrices and $B$ the Borel subgroup of $G$ consisting of upper triangular matrices. For $\mu$ a dominant weight of $T(\R\otimes_{\Q}F)$ with respect to $B(\R\otimes_{\Q}F)$, we can define $W_{\mu}$ an irreducible representation of $GL_{n}(\R\otimes_{\Q}F)$ with highest weight $\mu$.

From the proof of Th\'eor\`eme $3.13$ \cite{clozel}, we know that there exists a dominant algebraic weight $\mu$, such that $H^{*}(\lieG_{\infty},K_{\infty}^{0};\Pi_{\infty}\otimes W_{\mu})\neq 0$. 

Let $b$ be the smallest degree such that $H^{b}(\lieG_{\infty},K_{\infty}^{0};\Pi_{\infty}\otimes W_{\mu})\neq 0$. We have an explicit formula for $b$ in \cite{raghuramshahidi}. More precisely, we set $r_{1}$ and $r_{2}$ the numbers of real and complex embeddings of $F$ respectively. We have $b=r_{1}\left[\frac{n^{2}}{4}\right]+r_{2}\frac{n(n-1)}{2}$.

We can decompose this cohomology group via the action of $K_{\infty}/K_{\infty}^{0}$. There exists a character $\epsilon$ of $K_{\infty}/K_{\infty}^{0}$ described explicitly in \cite{raghuramshahidi} such that:
\begin{enumerate}
\item The isotypic component $H^{b}(\lieG_{\infty},K_{\infty}^{0};\Pi_{\infty}\otimes W_{\mu})(\epsilon)$ is one dimensional.

\item For fixed $w_{\infty}$, a generator of $H^{b}(\lieG_{\infty},K_{\infty}^{0};\Pi_{\infty}\otimes W_{\mu})(\epsilon)$, we have a $GL_{n}(\AF\text{}_{,f})$-equivariant isomorphisms:
\begin{eqnarray}\label{rational isomorphism}
W(\Pi_{f}) &\xrightarrow{\sim}& W(\Pi_{f})\otimes H^{b}(\lieG_{\infty},K_{\infty}^{0};\Pi_{\infty}\otimes W_{\mu})(\epsilon)\nonumber\\
&\xrightarrow{\sim}& H^{b}(\lieG_{\infty},K_{\infty}^{0};W(\Pi)\otimes W_{\mu})(\epsilon)\nonumber\\
&\xrightarrow{\sim}&  H^{b}(\lieG_{\infty},K_{\infty}^{0};\Pi \otimes W_{\mu})(\epsilon)
\end{eqnarray}
where the first map sends $w_{f}$ to $w_{f}\otimes w_{\infty}$ and the last map is given by the isomorphism $W(\Pi)\xrightarrow{\sim} \Pi$.  

\item The cohomology space $H^{b}(\lieG_{\infty},K_{\infty}^{0};\Pi \otimes W_{\mu})(\epsilon)$ is related to the cuspidal cohomology if $\Pi$ is cuspidal and to the Eisenstein cohomology if $\Pi$ is not cuspidal. In both cases, it is endowed with a $\Q(\Pi)$-rational structure (see \cite{raghuramshahidi} for cuspidal case and \cite{harrismotivic} for non cuspidal case).

\end{enumerate}

We denote by $\Theta_{\Pi_{f},\epsilon,w_{\infty}}$ the $GL_{n}(\AFf)$-isomorphism given in (\ref{rational isomorphism}) 
\begin{equation}W(\Pi_{f})\xrightarrow{\sim}  H^{b}(\lieG_{\infty},K_{\infty}^{0};\Pi \otimes W_{\mu})(\epsilon)\nonumber.
\end{equation}
 Both sides have a $\Q(\Pi)$-rational structure. In particular, the preimage of the rational structure on the right hand side gives a rational structure on $W(\Pi_{f})$. But the rational structure on $W(\Pi_{f})$ is unique up to homotheties. Therefore, there exists a complex number $p(\Pi_{f},\epsilon,w_{\infty})$ such that the new map $\Theta^{0}_{\Pi_{f},\epsilon,w_{\infty}}=p(\Pi_{f},\epsilon,w_{\infty})^{-1}\Theta_{\Pi_{f},\epsilon,w_{\infty}}$ preserves the rational structure on both sides. It is easy to see that this number $p(\Pi_{f},\epsilon,w_{\infty})$ is unique up to multiplication by elements in $\Q(\Pi)^{\times}$.

Finally, we observe that the $Aut(\C)$-action preserves rational structures on both the Whittaker models and cohomology spaces. We can adjust the numbers $p(\Pi_{f}^{\sigma},\epsilon^{\sigma},w_{\infty}^{\sigma})$ for all $\sigma\in Aut(\C)$ by elements in $\Q(\Pi)^{\times}$ such that the following diagram commutes:

\[\xymatrixcolsep{12pc}
\xymatrix{
W(\Pi_{f}) \ar[d]^{\sigma} \ar[r]^{p(\Pi_{f},\epsilon,w_{\infty})^{-1}\Theta_{\Pi_{f},\epsilon,w_{\infty}}} &H^{b}(\lieG_{\infty},K_{\infty}^{0};\Pi \otimes W_{\mu})(\epsilon)\ar[d]^{\sigma}\\
W(\Pi_{f}^{\sigma}) \ar[r]^{p(\Pi_{f}^{\sigma},\epsilon^{\sigma},w_{\infty}^{\sigma})^{-1}\Theta_{\Pi_{f}^{\sigma},\epsilon^{\sigma},w_{\infty}^{\sigma}}}          &H^{b}(\lieG_{\infty},K_{\infty}^{0};\Pi^{\sigma} \otimes W_{\mu}^{\sigma})(\epsilon^{\sigma})}
\]
The proof is the same as the cuspidal case in \cite{raghuramshahidi}.

In the following, we fix $\epsilon$, $w_{\infty}$ and we define the \textbf{Whittaker period} $p(\Pi):=p(\Pi_{f},\epsilon,w_{\infty})$. For any $\sigma\in Aut(\C)$, we define $p(\Pi^{\sigma}):=p(\Pi_{f}^{\sigma},\epsilon^{\sigma},w_{\infty}^{\sigma})$. It is easy to see that $p(\Pi^{\sigma})=p(\Pi)$ for $\sigma\in Aut(\C/\Q(\Pi))$. 

Moreover, the elements $(p(\Pi^{\sigma}))_{\sigma\in Aut(\C)}$ are well defined up to $\Q(\Pi)^{\times}$ in the following sense: if $(p'(\Pi^{\sigma}))_{\sigma\in Aut(\C)}$ is another family of complex numbers such that $p'(\Pi^{\sigma})^{-1}\Theta_{\Pi_{f}^{\sigma},\epsilon^{\sigma},w_{\infty}^{\sigma}}$ preserves the rational structure and the above diagram commutes, then there exists $t\in \Q(\Pi)^{\times}$ such that $p'(\Pi^{\sigma})=\sigma(t) p(\Pi^{\sigma})$ for any $\sigma\in Aut(\C)$. This also follows from the one dimensional property of the invariant vector. The argument is the same as the last part of the proof of Definition/Proposition $3.3$ in \cite{raghuramshahidi}.\\

The Whittaker periods are closely related to special values of $L$-functions. We refer to \cite{raghuram10}, \cite{harderraghuram}, \cite{harrismotivic} or \cite{grob_harris_lapid} for more details. Here we state a theorem which generalizes the main theorem of \cite{harrismotivic}. The proof can be found in \cite{haraldappendix}.

 Let $\Pi$ be a regular cuspidal cohomological representation of $GL_{n}(\AF)$. Let $\Pi^{\#}$ be a regular automorphic cohomological representation of $GL_{n-1}(\AF)$ which is the Langlands sum of cuspidal representations. Write the infinity type of $\Pi$ (resp. $\Pi'$) at $\sigma\in \Sigma$ by $\{z^{a_{i}(\sigma)}\bar{z}^{-a_{i}(\sigma)}\}_{1\leq i\leq n}$ (resp $\{z^{b_{j}(\sigma)}\bar{z}^{-b_{j}(\sigma)}\}_{1\leq j\leq n-1}$) with $a_{1}(\sigma)> a_{2}(\sigma)> \cdots >a_{n}(\sigma)$ (resp. $b_{1}(\sigma)> b_{2}(\sigma)> \cdots >b_{n-1}(\sigma)$). We say that the pair $(\Pi,\Pi^{\#})$ is in \textit{good position} if for all $\sigma\in\Sigma$ we have $$a_{1}(\sigma)>-b_{n-1}(\sigma)>a_{2}(\sigma)>\cdots>-b_{1}(\sigma)>a_{n}(\sigma).$$

 \begin{thm}\label{Whittaker period theorem CM}
If $(\Pi,\Pi^{\#})$ is in good position, then there exists a non-zero complex number $p(m,\Pi_{\infty},\Pi^{\#}_{\infty})$ which depends on $m,\Pi_{\infty}$ and $\Pi^{\#}_{\infty}$ well defined up to $(E(\Pi)E(\Pi^{\#}))^{\times}$ such that for $m\in \Z$ with $m+\cfrac{1}{2}$ critical for $\Pi\times \Pi^{\#}$, we have
\begin{equation}\label{CM Whittaker period}
L(\cfrac{1}{2}+m,\Pi\times \Pi^{\#})\sim_{E(\Pi)E(\Pi^{\#})} p(m,\Pi_{\infty},\Pi^{\#}_{\infty})p(\Pi)p(\Pi^{\#})
\end{equation}
and is equivariant under the action of $Aut(\C/\Q)$.
\end{thm}

\section{Arithmetic automorphic periods}
\subsection{CM periods}\label{section CM periods}
Let $(T,h)$ be a Shimura datum where $T$ is a torus defined over $\Q$ and $h:Res_{\C/\R}\mathbb{G}_{m,\C}\rightarrow G_{\R}$ a homomorphism satisfying the axioms defining a Shimura variety. Such pair is called a \textbf{special} Shimura datum. Let $Sh(T,h)$ be the associated Shimura variety and $E(T,h)$ be its reflex field.

Let $(\gamma,V_{\gamma})$ be a one-dimensional algebraic representation of $T$ (the representation $\gamma$ is denoted by $\chi$ in \cite{harrisappendix}). We denote by $E(\gamma)$ a definition field for $\gamma$. We may assume that $E(\gamma)$ contains $E(T,h)$. Suppose that $\gamma$ is motivic (see \textit{loc.cit} for the notion). We know that $\gamma$ gives an automorphic line bundle $[V_{\gamma}]$ over $Sh(T,h)$ defined over $E(\gamma)$. Therefore, the complex vector space $H^{0}(Sh(T,h),[V_{\gamma}])$ has an $E(\gamma)$-rational structure, denoted by $M_{DR}(\gamma)$ and called the De Rham rational structure.

On the other hand, the canonical local system $V_{\gamma}^{\triangledown}\subset [V_{\gamma}]$ gives another $E(\gamma)$-rational structure $M_{B}(\gamma)$ on $H^{0}(Sh(T,h),[V_{\gamma}])$, called the Betti rational structure.

We now consider $\chi$ an algebraic Hecke character of $T(\AQ)$ with infinity type $\gamma^{-1}$ (our character $\chi$ corresponds to the character $\omega^{-1}$ in \textit{loc.cit}). Let $E(\chi)$ be the number field generated by the values of $\chi$ on $T(\AQ \text{}_{,f})$ over $E(\gamma)$. We know $\chi$ generates a one-dimensional complex subspace of $H^{0}(Sh(T,h),[V_{\gamma}])$ which inherits two $E(\chi)$-rational structures, one from $M_{DR}(\gamma)$, the other from $M_{B}(\gamma)$. Put $p(\chi,(T,h))$ the ratio of these two rational structures which is well defined modulo $E(\chi)^{\times}$.

\begin{rem}\label{CM complex conjugation}
If we identify $H^{0}(Sh(T,h),[V_{\gamma}])$ with the set $$\{f\in \mathbb{C}^{\infty}(T(\Q)\backslash T(\AQ),\C\mid f(tt_{\infty}))=\gamma^{-1}(t_{\infty})f(t), t_{\infty}\in T(\R), t\in T(\AQ)\}$$, then $\chi$ itself is in the rational structure inherits from $M_{B}(\gamma)$. See discussion from $A.4$ to $A.5$ in \cite{harrisappendix}.
\end{rem}

Suppose that we have two tori $T$ and $T'$ both endowed with a Shimura datum $(T,h)$ and $(T',h')$. Let $u:(T',h')\rightarrow (T,h)$ be a map between the Shimura data. Let $\chi$ be an algebraic Hecke character of $T(\AQ)$. We put $\chi':=\chi\circ u$ an algebraic Hecke character of $T'(\AQ)$. Since both the Betti structure and the De Rham structure commute with the pullback map on cohomology, we have the following proposition:

\begin{prop}\label{propgeneral}
Let $\chi$, $(T,h)$ and $\chi'$, $(T',h')$ be as above. We have:
\begin{equation}\nonumber
p(\chi,(T,h)) \sim_{E(\chi)} p(\chi',(T',h'))
\end{equation}
and is equivariant under the action of $Aut(\C/E(T)E(T'))$.
\end{prop}

\begin{rem}
In Proposition $1.4$ of \cite{harrisCMperiod}, the relation is up to $E(\chi);E(T,h)$ where $E(T,h)$ is a number field associated to $(T,h)$. Here we consider the action of $G_{\Q}$ and can thus obtain a relation up to $E(\chi)$ (see the paragraph after Proposition $1.8.1$ of \textit{loc.cit}).
\end{rem}

For $F$ a CM field and $\Psi$ a subset of $\Sigma_{F}$ such that $\Psi\cap \iota\Psi=\varnothing$, we can define a Shimura datum $(T_{F},h_{\Psi})$ where $T_{F}:=Res_{F/\Q}\mathbb{G}_{m,F}$ is a torus and $h_{\Psi}:Res_{\C/\R}\mathbb{G}_{m,\C} \rightarrow T_{F,\R}$ is a homomorphism such that over $\sigma\in \Sigma_{F}$, the Hodge structure induced on $F$ by $h_{\Psi}$ is of type $(-1,0)$ if $\sigma\in \Psi$, of type $(0,-1)$ if $\sigma\in \iota\Psi$, and of type $(0,0)$ otherwise. \\

Let $\chi$ be a motivic critical character of a CM field $F$. By definition, $p_{F}(\chi,\Psi)=p(\chi,(T_{F},h_{\Psi}))$ and we call it a \textbf{CM period}. Sometimes we write $p(\chi,\Psi)$ instead of $p_{F}(\chi,\Psi)$ if there is no ambiguity concerning the base field $F$.

\begin{ex}
We have $p(||\cdot||_{\AK},1)\sim_{\Q} (2\pi i)^{-1}.$ See  $(1.10.9)$ on page $100$ of \cite{harris97}. 
\end{ex}

\begin{prop}\label{propCM}

Let $\tau\in J_{F}$ and $\Psi$ be a subset of $J_{F}$ such that $\Psi\cap \Psi^{c}=\varnothing$. We take $\Psi=\Psi_{1}\sqcup\Psi_{2}$ a partition of $\Psi$. 

 We then have:
\begin{eqnarray}\label{charmulti}
&p(\chi_{1}\chi_{2}), \Psi) \sim _{E(\chi_{1})E(\chi_{2})} p(\chi_{1}, \Psi)p(\chi_{2}^{\sigma}, \Psi).&
\\ \label{separateCMtype}
&p(\chi, \Psi_{1}\sqcup\Psi_{2})\sim _{E(\chi) } p(\chi, \Psi_{1})p(\chi,\Psi_{2}).&
\\
&p(\chi, \Psi) \sim _{E(\chi)} p(\bar{\chi},\bar{\Psi}).&
\end{eqnarray}
All the relations are equivariant under the action of $Aut(\C/F^{Gal})$.
\end{prop}

\begin{proof}
All the equations in Proposition \ref{propCM} come from Proposition \ref{propgeneral} by certain maps between Shimura data as follows:
\begin{enumerate}
\item The diagonal map $ (T_{F},h_{\Psi})\rightarrow (T_{F}\times T_{F},h_{\Psi}\times h_{\Psi})$ pulls $(\chi_{1},\chi_{2})$ back to $\chi_{1}\chi_{2}$.
\item The multiplication map $T_{F} \times T_{F} \rightarrow T_{F}$ sends $h_{\Psi_{1}}$, $h_{\Psi_{2}}$ to $h_{\Psi_{1}\sqcup \Psi_{2}}$.
\item The Galois action $\theta: H_{F}\rightarrow H_{F}$ sends $h_{\Psi}$ to $h_{\Psi^{\theta}}$.
\item The norm map $(T_{F},h_{\{\tau\}})\rightarrow (T_{F_{0}},h_{\{\tau|_{F_{0}}\}})$ pulls $\eta$ back to $\eta\circ N_{\AF/\mathbb{A}_{F,0}}$.
\end{enumerate}
\end{proof}

The special values of an $L$-function for a Hecke character over a CM field can be interpreted in terms of CM periods. The following theorem is proved by Blasius. We state it as in Proposition $1.8.1$ in \cite{harrisCMperiod} where $\omega$ should be replaced by $\check{\omega}:=\omega^{-1,c}$ (for this erratum, see the notation and conventions part on page $82$ in the introduction of \cite{harris97}),

\begin{thm}\label{blasius}
Let $F$ be a CM field and $F^{+}$ be its maximal totally real subfield. Put $d$ the degree of $F^{+}$ over $\Q$.

Let $\chi$ be a motivic critical algebraic Hecke character of $F$ and $\Phi_{\chi}$ be the unique CM type of $F$ which is compatible with $\chi$.

For $m$ a critical value of $\chi$ in the sense of Deligne (c.f. \cite{deligne79}), we have
\begin{equation}\nonumber
L(\chi,m) \sim_{E(\chi)} (2\pi i)^{md}p(\check{\chi},\Phi_{\chi})
\end{equation}
equivariant under action of $Aut(\C/F^{Gal})$.
\end{thm}

\begin{rem}\label{criticalforcharacter}
 Let $\{\sigma_{1},\sigma_{2},\cdots,\sigma_{n}\}$ be any CM type of $F$. Let $(\sigma_{i}^{a_{i}}\overline{\sigma}_{i}^{-w-a_{i}})_{1\leq i\leq n}$ denote the infinity type of $\chi$ with $w=w(\chi)$. We may assume $a_{1}\geq a_{2}\geq \cdots \geq a_{n}$. We define $a_{0}:=+\infty$ and $a_{n+1}:=-\infty$ and define $k:=max\{0\leq i\leq n\mid a_{i}>-\cfrac{w}{2}\}$. An integer $m$ is critical for $\chi$ if and only if

\begin{equation}\label{criticalvalue1}
       max(-a_{k}+1,w+1+a_{k+1})\leq m \leq min(w+a_{k},-a_{k+1}).
\end{equation}

\end{rem}

\subsection{Construction of cohomology spaces}

Let $I$ be the signature of a unitary group $U_{I}$ of dimension $n$ with respect to the extension $F/F^{+}$. Let $V_{I}$ be the associated Hermitian space. We can consider $I$ as a map from $\Sigma$ to $\{0,1,\cdots,n\}$. We write $s_{\sigma}:=I(\sigma)$ and $r_{\sigma}:=n-I(\sigma)$ for all $\sigma\in\Sigma$.

Denote $\Ss:=Res_{\C/\R}\mathbb{G}_{m}$. We define the rational similitude unitary group defined by
\begin{equation}\label{definition of nu(g)}
GU_{I}(R):=\{g\in GL(V_{I}\otimes_{\Q}R)|(gv,gw)=\nu(g)(v,w), 
  \nu(g)\in R^{*}\}
  \end{equation}
 where $R$ is any $\Q$-algebra.

We know that $GU_{I}(\R)$ is isomorphic to a subgroup of $\prod\limits_{\sigma\in \Sigma}GU(r_{\sigma},s_{\sigma})$ defined by the same \textit{similitude}. We can define a homomorphism $h_{I}:\Ss(\R) \rightarrow GU_{I}(\R)$ by sending $z\in \C$ to $\left(\begin{pmatrix}
zI_{r_{\sigma}} & 0\\
0 & \bar{z}I_{s_{\sigma}}
\end{pmatrix}\right)_{\sigma\in \Sigma}$. \\

Let $X_{I}$ be the $GU_{I}(\R)$-conjugation class of $h_{I}$. We know $(GU_{I},X_{I})$ is a Shimura datum with reflex field $E_{I}$ and dimension $2\sum\limits_{\sigma\in \Sigma} r_{\sigma}s_{\sigma}$. The Shimura variety associated to $(GU_{I},X_{I})$ is denoted by $Sh_{I}$.

Let $K_{I,\infty}$ be the centralizer of $h_{I}$ in $GU_{I}(\R)$. Via the inclusion $GU_{I}(\R)\hookrightarrow \prod\limits_{\sigma\in \Sigma}GU(r_{\sigma},s_{\sigma})\subset \R^{+,\times}\prod\limits_{\sigma\in\Sigma} U(n,\C)$, we may identify $K_{I,\infty}$ with \begin{equation}\nonumber   \{  (\mu,\begin{pmatrix}
 u_{r_{\sigma}} & 0\\
0 &  v_{s_{\sigma}}
\end{pmatrix}_{\sigma\in \Sigma}) \mid u_{r_{\sigma}}\in U(r_{\sigma},\C), v_{s_{\sigma}}\in U(s_{\sigma},\C), \mu\in \R^{+,\times}\}\end{equation}  
where $U(r,\C)$ is the standard unitary group of degree $r$ over $\C$. Let $H_{I}$ be the subgroup of $K_{I,\infty}$ consisting of the diagonal matrices in $K_{I,\infty}$. Then it is a maximal torus of $GU_{I}(\R)$. Denote its Lie algebra by $\lieH_{I}$.

We observe that $H_{I}(\R) \cong \R^{+,\times}\times \prod\limits_{\sigma\in\Sigma} U(1,\C)^{n}$. Its algebraic characters are of the form 
\begin{equation}
\nonumber   (w,(z_{i}(\sigma))_{\sigma\in\Sigma,1\leq i\leq n})\mapsto w^{\lambda_{0}} \prod\limits_{\sigma\in\Sigma} \prod\limits_{i=1}^{n}z_{i}(\sigma)^{\lambda_{i}(\sigma)}
\end{equation}  
 where $(\lambda_{0},(\lambda_{i}(\sigma))_{\sigma\in\Sigma,1\leq i\leq n})$ is a $(nd+1)$-tuple of integers with $\lambda_{0} \equiv \sum\limits_{\sigma\in\Sigma}\sum\limits_{i=1}^{n}\lambda_{i}(\sigma)$ (mod $2$).

 Recall that $GU_{I}(\C)\cong \C^{\times}\prod_{\sigma\in\Sigma}GL_{n}(\C)$. We fix $B_{I}$ the Borel subgroup of $GU_{I,\C}$ consisting of upper triangular matrices. The highest weights of finite-dimensional irreducible representations of $K_{I,\infty}$ are tuples $\Lambda=(\Lambda_{0},(\Lambda_{i}(\sigma))_{\sigma\in\Sigma,1\leq i\leq n})$ such that $\Lambda_{1}(\sigma)\geq \Lambda_{2}(\sigma)\geq \cdots\geq \Lambda_{r_{\sigma}}(\sigma)$, $ \Lambda_{r_{\sigma}+1}(\sigma) \geq \cdots \geq  \Lambda_{n}(\sigma)$ for all $\sigma$ and $\Lambda_{0} \equiv \sum\limits_{\sigma\in\Sigma}\sum\limits_{i=1}^{n}\Lambda_{i}(\sigma)$ (mod $2$).

We denote the set of such tuples by $\Lambda(K_{I,\infty})$. Similarly, we write $\Lambda(GU_{I})$ for the set of the highest weights of finite-dimensional irreducible representations of $GU_{I}$. It consists of tuples $\lambda=(\lambda_{0},(\lambda_{i}(\sigma))_{\sigma\in\Sigma,1\leq i\leq n})$ such that $\lambda_{1}(\sigma)\geq \lambda_{2}(\sigma) \geq  \cdots \lambda_{n}(\sigma)$ for all $\sigma$ and $\lambda_{0} \equiv \sum\limits_{\sigma\in\Sigma}\sum\limits_{i=1}^{n}\lambda_{i}(\sigma)$ (mod $2$).

We take $\lambda\in\Lambda(GU_{I})$ and $\Lambda\in \Lambda(K_{I,\infty})$.

Let $V_{\lambda}$  and $V_{\Lambda}$ be the corresponding representations.  We define a local system over $Sh_{I}$:
\begin{equation}\nonumber   W_{\lambda}^{\triangledown}:=\lim_{\overleftarrow{K}} GU_{I}(\Q)\backslash V_{\lambda}\times X\times GU_{I}(\AQf)/K\end{equation}  
and an automorphic vector bundle over $Sh_{I}$
\begin{equation}\nonumber   E_{\Lambda}:=\lim_{\overleftarrow{K}} GU_{I}(\Q)\backslash V_{\Lambda}\times GU_{I}(\R) \times GU_{I}(\AQf)/KK_{I,\infty}\end{equation}  
where $K$ runs over open compact subgroup of $GU_{I}(\AQf)$.

The automorphic vector bundles $E_{\Lambda}$ are defined over the reflex field $E$. \\

The local systems $W_{\lambda}^{\triangledown}$ are defined over $K$. The Hodge structure of the cohomology space $H^{q}(Sh_{I},W_{\lambda}^{\triangledown})$ is not pure in general. But the image of $H_{c}^{q}(Sh_{I},W_{\lambda}^{\triangledown})$ in $H^{q}(Sh_{I},W_{\lambda}^{\triangledown})$ is pure of weight $q-c$. We denote this image by $\bar{H}^{q}(Sh_{I},W_{\lambda}^{\triangledown})$. \\

Note that all cohomology spaces have coefficients in $\C$ unless we specify its rational structure over a number field.

\subsection{The Hodge structures}
\text{}
The results in section $2.2$ of \cite{harrismotives} give a description of the Hodge components of $\bar{H}^{q}(Sh_{I},W_{\lambda}^{\triangledown})$.

Denote by $R^{+}$ the set of positive roots of $H_{I,\C}$ in $GU_{I}(\C)$ and by $R_{c}^{+}$ the set of positive compact roots. Define $\alpha_{j,k}=(0,\cdots,0,1,0,\cdots,0,-1,0,\cdots,0)$ for any $1\leq j<k\leq n$. We know $R^{+}=\{(\alpha_{j_{\sigma},k_{\sigma}})_{\sigma\in \Sigma}\mid 1\leq j_{\sigma}<k_{\sigma}\leq n\}$ and 
$R_{c}^{+}=\{(\alpha_{j_{\sigma},k_{\sigma}})_{\sigma\in \Sigma}\mid j_{\sigma}<k_{\sigma}\leq r_{\sigma} \text{ or } r_{\sigma}+1\leq j_{\sigma}<k_{\sigma}\}$.

Let $\rho=\cfrac{1}{2}\sum\limits_{\alpha\in R^{+}}\alpha=\left(\left(\cfrac{n-1}{2},\cfrac{n-3}{2},\cdots,-\cfrac{n-1}{2}\right)\right)_{\sigma}$. \\

Let $\lieG$, $\lieK$ and $\lieH$ be Lie algebras of $GU_{I}(\R)$, $K_{I,\infty}$ and $H(\R)$. Write $W$ for the Weyl group $W(\lieG_{\C},\lieH_{\C})$ and $W_{c}$ for the Weyl group $W(\lieK_{\C},\lieH_{\C})$. We can identify $W$ with $\prod\limits_{\sigma\in\Sigma}\mathfrak{S}_{n}$ and $W_{c}$ with $\prod\limits_{\sigma\in\Sigma}\mathfrak{S}_{r_{\sigma}}\times \mathfrak{S}_{s_{\sigma}}$ where $\mathfrak{S}$ refers to the standard permutation group. For $w\in W$, we write the length of $w$ by $l(w)$.

Let $W^{1}:=\{w\in W| w(R^{+})\supset R_{c}^{+}\}$ be a subset of $W$. By the above identification, $(w_{\sigma})_{\sigma}\in W^{1}$ if and only if  $w_{\sigma}(1)<w_{\sigma}(2)<\cdots<w_{\sigma}(r_{\sigma})$ and $w_{\sigma}(r_{\sigma}+1)<\cdots<w_{\sigma}(n)$  One can show that $W^{1}$ is a set of coset representatives of shortest length for $W_{c}\backslash W$.

Moreover, for $\lambda$ a highest weight of a representation of $GU_{I}$, one can show easily that $w*\lambda:=w(\lambda+\rho)-\rho$ is the highest weight of a representation of $K_{I,\infty}$. More precisely, if $\lambda=(\lambda_{0},(\lambda_{i}(\sigma))_{\sigma\in\Sigma,1\leq i\leq n})$, then $w*\lambda=(\lambda_{0},((w*\lambda)_{i}(\sigma))_{\sigma\in\Sigma,1\leq i\leq n})$ with $(w*\lambda)_{i}(\sigma)=\lambda_{w_{\sigma}(i)}(\sigma)+\cfrac{n+1}{2}-w_{\sigma}(i)-(\cfrac{n+1}{2}-i)=\lambda_{w_{\sigma}(i)}(\sigma)-w_{\sigma}(i)+i$.

\begin{rem}\label{hodgetype}
The results of \cite{harrismotives} tell us that there exists
\begin{equation}
\bar{H}^{q}(Sh_{I},W_{\lambda}^{\triangledown})\cong \bigoplus\limits_{w\in W^{1}}\bar{H}^{q;w}(Sh_{I},W_{\lambda}^{\triangledown})
\end{equation}
a decomposition as subspaces of pure Hodge type $(p(w,\lambda),q-c-p(w,\lambda))$. We now determine the Hodge number $p(w,\lambda)$.

We know that $w*\lambda$ is the highest weight of a representation of $K_{I,\infty}$. We denote this representation by $(\rho_{w*\lambda},W_{w*\lambda})$. We know that $\rho_{w*\lambda} \circ h_{I}|_{\Ss(\R)}: \Ss(\R)\rightarrow K_{I,\infty} \rightarrow GL(W_{w*\lambda})$ is of the form $z\mapsto z^{-p(w,\lambda)}\bar{z}^{-r(w,\lambda)}I_{W_{w*\lambda}}$ with $p(w,\lambda)$, $r(w,\lambda)\in \Z$. The first index $p(w,\lambda)$ is the Hodge type mentioned above. 

Recall that the map 
\begin{eqnarray}
h_{I}|_{\Ss(\R)}: \Ss(\R)&\rightarrow& K_{I,\infty}\subset \R^{+,\times}\times U(n,\C)^{\Sigma}\\ z&\mapsto& \left(|z|,
\begin{pmatrix}
\cfrac{z}{|z|}I_{r_{\sigma}} & 0\\\nonumber
0 & \cfrac{\bar{z}}{|z|}I_{s_{\sigma}}
\end{pmatrix}_{\sigma\in\Sigma}\right)
\end{eqnarray}
and the map 
\begin{eqnarray}
\nonumber \rho_{w*\lambda}: K_{I,\infty} &\rightarrow& GL(W_{w*\lambda}) \\\nonumber (w,diag(z_{i}(\sigma))_{\sigma\in\Sigma,1\leq i\leq n})&\mapsto & w^{\lambda_{0}}\prod\limits_{\sigma\in\Sigma} \prod\limits_{i=1}^{n}z_{i}(\sigma)^{(w*\lambda)_{i}(\sigma)}
\end{eqnarray}
 where $diag(z_{1},z_{2},\cdots,z_{n})$ means the diagonal matrix of coefficients $z_{1},z_{2},\cdots,z_{n}$.

Therefore we have:

\begin{eqnarray}
&&z^{-p(w,\lambda)}\bar{z}^{-r(w,\lambda)} \nonumber\\
&=&|z|^{\lambda_{0}}\prod\limits_{\sigma\in\Sigma}(\prod\limits_{1\leq i\leq r_{\sigma}}(\cfrac{z}{|z|})^{(w*\lambda)_{i}(\sigma)}\prod\limits_{r_{\sigma}+1\leq i\leq n}(\cfrac{\overline{z}}{|z|})^{(w*\lambda)_{i}(\sigma)}\nonumber\\
&=& (z^{\frac{1}{2}}\overline{z}^{\frac{1}{2}})^{\lambda_{0}-\sum\limits_{\sigma\in\Sigma}\sum\limits_{1\leq i\leq n}(w*\lambda)_{i}(\sigma)}z^{\sum\limits_{\sigma\in\Sigma}\sum\limits_{1\leq i\leq r_{\sigma}}(w*\lambda)_{i}(\sigma)}
\overline{z}^{\sum\limits_{\sigma\in\Sigma}\sum\limits_{r_{\sigma}+1\leq i\leq n}(w*\lambda)_{i}(\sigma)} \nonumber
\end{eqnarray}

Since $(w*\lambda)_{i}(\sigma)=\lambda_{w_{\sigma}(i)}(\sigma)-w_{\sigma}(i)+i$ and then $\sum\limits_{\sigma\in\Sigma}\sum\limits_{1\leq i\leq n}(w*\lambda)_{i}(\sigma)=\sum\limits_{\sigma\in\Sigma}\sum\limits_{1\leq i\leq n}\lambda_{i}(\sigma)$, we obtain that:
\begin{eqnarray}\label{smallest}
p(w,\lambda) &=& \cfrac{\sum\limits_{\sigma\in\Sigma}\sum\limits_{1\leq i\leq n}\lambda_{i}(\sigma)-\lambda_{0}}{2}-\sum\limits_{\sigma\in\Sigma}\sum\limits_{1\leq i\leq r_{\sigma}}(w*\lambda)_{i}(\sigma)\nonumber\\
&=& \cfrac{\sum\limits_{\sigma\in\Sigma}\sum\limits_{1\leq i\leq n}\lambda_{i}(\sigma)-\lambda_{0}}{2}-\sum\limits_{\sigma\in\Sigma}\sum\limits_{1\leq i\leq r_{\sigma}}(\lambda_{w_{\sigma}(i)}(\sigma)-w_{\sigma}(i)+i)
\end{eqnarray}
\end{rem}

The method of toroidal compactification gives us more information on $\bar{H}^{q;w}(Sh_{I},W_{\lambda}^{\triangledown})$. We take $j:Sh_{I} \hookrightarrow \widetilde{Sh}_{I}$ to be a smooth toroidal compactification. Proposition $2.2.2$ of \cite{harrismotives} tells us that the following results do not depend on the choice of the toroidal compactification.

The automorphic vector bundle $E_{\Lambda}$ can be extended to $\widetilde{Sh}_{I}$ in two ways: the canonical extension $E_{\Lambda}^{can}$ and the sub canonical extension $E_{\Lambda}^{sub}$ as explained in \cite{harrismotives}. Define:
\begin{equation}\nonumber   \bar{H}^{q}(Sh_{I},E_{\Lambda})=Im(H^{q}(\widetilde{Sh}_{I},E_{\Lambda}^{sub})\rightarrow H^{q}(\widetilde{Sh}_{I},E_{\Lambda}^{can})).\end{equation}

\begin{prop}
There is a canonical isomorphism
\begin{equation}\nonumber   \bar{H}^{q;w}(Sh_{I},W_{\lambda}^{\triangledown})\cong \bar{H}^{q-l(w)}(Sh_{I},E_{w*\lambda})\end{equation}  
\end{prop}

Let $D=2\sum\limits_{\sigma\in\Sigma}r_{\sigma}s_{\sigma}$ be the dimension of the Shimura variety. We are interested in the cohomology space of degree $D/2$. Proposition $2.2.7$ of \cite{harris97} also works here:

\begin{prop}
The space $\bar{H}^{D/2}(Sh_{I},W_{\lambda}^{\triangledown})$ is naturally endowed with a $K$-rational structure, called the de Rham rational structure and noted by $\bar{H}_{DR}^{D/2}(Sh_{I},W_{\lambda}^{\triangledown})$. This rational structure is endowed with a $K$-Hodge filtration $F^{\cdot}\bar{H}_{DR}^{D/2}(Sh_{I},W_{\lambda}^{\triangledown})$ pure of weight $D/2-c$ such that
\begin{equation}\nonumber   F^{p}\bar{H}_{DR}^{D/2}(Sh_{I},W_{\lambda}^{\triangledown})/F^{p+1}\bar{H}_{DR}^{D/2}(Sh_{I},W_{\lambda}^{\triangledown})\otimes_{K} \C \cong \bigoplus\limits_{w\in W^{1},p(w,\lambda)=p} \bar{H}^{D/2;w}(Sh_{I},W_{\lambda}^{\triangledown}).\end{equation}  
Moreover, the composition of the above isomorphism and the canonical isomorphism \begin{equation}\nonumber
\bar{H}^{D/2;w}(Sh_{I},W_{\lambda}^{\triangledown})\cong \bar{H}^{D/2-l(w)}(Sh_{I},E_{w*\lambda})
\end{equation} is rational over $K$.
\end{prop}

\paragraph{Holomorphic part:}

Let $w_{0}\in W^{1}$ defined by $$w_{0}(\sigma)(1,2,\cdots,r_{\sigma};r_{\sigma+1},\cdots,n)_{\sigma\in\Sigma}=(s_{\sigma+1},\cdots,n;1,2,\cdots,s_{\sigma})$$ for all $\sigma\in\Sigma$. It is the only longest element in $W^{1}$. Its length is $D/2$. 

We have a $K$-rational isomorphism \begin{equation}\bar{H}^{D/2;w_{0}}(Sh_{I},W_{\lambda}^{\triangledown})\cong \bar{H}^{0}(Sh_{I},E_{w_{0}*\lambda}).
\end{equation}

We can calculate the Hodge type of $\bar{H}^{D/2;w_{0}}(Sh_{I},W_{\lambda}^{\triangledown})$ as in Remark \ref{hodgetype}.

By definition we have \begin{equation}\label{twist by w0}
w_{0}*\lambda=(\lambda_{0},(\lambda_{s_{\sigma}+1}(\sigma)-s_{\sigma},\cdots,\lambda_{n}(\sigma)-s_{\sigma};\lambda_{1}(\sigma)+r_{\sigma},\cdots,\lambda_{s_{\sigma}}(\sigma)+r_{\sigma})_{\sigma\in\Sigma}).
\end{equation} By the discussion in Remark \ref{hodgetype}, the Hodge number 

\begin{equation}\nonumber   
p(w_{0},\lambda)=\cfrac{\sum\limits_{\sigma\in\Sigma}\sum\limits_{1\leq i\leq n}\lambda_{i}(\sigma)-\lambda_{0}+D}{2}-\sum\limits_{\sigma\in\Sigma}(\lambda_{s_{\sigma}+1}(\sigma)+\cdots+\lambda_{n}(\sigma)).\end{equation}  

From equation (\ref{smallest}), it is easy to deduce that $p(w_{0},\lambda)$ is the only largest number among $\{p(w,\lambda)\mid w\in W^{1}\}$. Therefore \begin{equation}F^{p(w_{0},\lambda)} (Sh_{I},W_{\lambda}^{\triangledown})\otimes_{K} \C \cong \bar{H}^{0}(Sh_{I},E_{w_{0}*\lambda}).
\end{equation} Moreover, as mentioned in the above proposition, we know that the above isomorphism is $K$-rational.

We call $\bar{H}^{D/2;w_{0}}(Sh_{I},W_{\lambda}^{\triangledown})\cong\bar{H}^{0}(Sh_{I},E_{w_{0}*\lambda})$ the \textbf{holomorphic part} of the Hodge decomposition of $\bar{H}^{D/2}(Sh_{I},W_{\lambda}^{\triangledown})$. It is isomorphic to the space of holomorphic cusp forms of type $(w_{0}*\lambda)^{\vee}$.

\paragraph{Anti-holomorphic part:}\text{}
The only shortest element in $W^{1}$ is the identity with the smallest Hodge number 
\begin{equation}\nonumber   
p(id,\lambda)=\cfrac{\sum\limits_{\sigma\in\Sigma}\sum\limits_{1\leq i\leq n}\lambda_{i}(\sigma)-\lambda_{0}}{2}-\sum\limits_{\sigma\in\Sigma}(\lambda_{1}(\sigma)+\cdots+\lambda_{r_{\sigma}}(\sigma)).
\end{equation} 
We call $\nonumber   \bar{H}^{D/2;id}(Sh_{I},W_{\lambda}^{\triangledown})\cong \bar{H}^{D/2}(Sh_{I},E_{\lambda})$
the \textbf{anti-holomorphic part} of the Hodge decomposition of $\bar{H}^{D/2}(Sh_{I},W_{\lambda}^{\triangledown})$.

\subsection{Complex conjugation}\label{Complex conjugation}

We specify some notation first. \\

Let $\lambda=(\lambda_{0},(\lambda_{1}(\sigma)\geq \lambda_{2}(\sigma)\geq \cdots \geq \lambda_{n}(\sigma))_{\sigma\in\Sigma}) \in \Lambda(GU_{I})$ as before. We define $\lambda^{c}:=(\lambda_{0},(-\lambda_{n}(\sigma)\geq -\lambda_{n-1}(\sigma)\geq \cdots \geq -\lambda_{1}(\sigma))_{\sigma\in\Sigma})$ and $\lambda^{\vee}:=(-\lambda_{0},(-\lambda_{n}(\sigma)\geq -\lambda_{n-1}(\sigma)\geq \cdots \geq -\lambda_{1}(\sigma))_{\sigma\in\Sigma})$. They are elements in $\Lambda(GU_{I})$. Moreover, the representation $V_{\lambda^{c}}$ is the complex conjugation of $V_{\lambda}$ and the representation $V_{\lambda^{\vee}}$ is the dual of $V_{\lambda}$ as $GU_{I}$-representation.

Similarly, for $\Lambda=(\Lambda_{0},(\Lambda_{1}(\sigma)\geq \cdots \geq \Lambda_{r_{\sigma}}(\sigma),\Lambda_{r_{\sigma}+1}(\sigma)\geq \cdots \geq \Lambda_{n}(\sigma))_{\sigma\in\Sigma})\in \Lambda(K_{I,\infty})$, we define $\Lambda^{*}:=(-\Lambda_{0},(-\Lambda_{r_{\sigma}}(\sigma)\geq \cdots \geq -\Lambda_{1}(\sigma),-\Lambda_{n}\geq\cdots\geq -\Lambda_{r_{\sigma}+1})_{\sigma\in\Sigma})$. 

We know $V_{\Lambda^{*}}$ is the dual of $V_{\Lambda}$ as $K_{I}$-representation. We sometimes write the latter as $\widecheck{V_{\Lambda}}$.

We define $I^{c}$ by $I^{c}(\sigma)=n-I(\sigma)$ for all $\sigma\in\Sigma$. We know $V_{I^{c}}=-V_{I}$ and $GU_{I^{c}}\cong GU_{I}$. The complex conjugation gives an anti-holomorphic isomorphism $X_{I}\xrightarrow{\sim} X_{I^{c}}$. This induces a $K$-antilinear isomorphism \begin{equation}
\bar{H}^{D/2}(Sh_{I},W_{\lambda}^{\triangledown})\xrightarrow{\sim}  \bar{H}^{D/2}(Sh_{I^{c}},W_{\lambda^{c}}^{\triangledown}).
\end{equation}
In particular, it sends holomorphic (resp. anti-holomorphic) elements with respect to $(I,\lambda)$ to those respect to $(I^{c},\lambda^{c})$. If we we denote by $w_{0}^{c}$ the longest element related to $I^{c}$ then we have $K$-antilinear rational isomorphisms
\begin{eqnarray}
c_{DR}:  &\bar{H}^{0}(Sh_{I},E_{w_{0}*\lambda})\xrightarrow{\sim}  \bar{H}^{0}(Sh_{I^{c}},E_{w_{0}^{c}*\lambda^{c}})&\\
& \bar{H}^{D/2}(Sh_{I},E_{\lambda})\xrightarrow{\sim}  \bar{H}^{D/2}(Sh_{I^{c}},E_{\lambda^{c}})&.
\end{eqnarray}

 The Shimura datum $(GU_{I},h)$ induces a Hodge structure of wights concentrated in $\{(-1,1),(0,0),(1,-1)\}$ which corresponds to the Harish-Chandra decomposition induced by $h$ on the Lie algebra:
$\lieG=\lieK_{\C}\oplus \lieP^{+}\oplus \lieP^{-}$.

Let $\mathfrak{P}=\lieK_{\C}\oplus \lieP^{-}$. Let $\mathcal{A}$ (resp. $\mathcal{A}_{0}$, $\mathcal{A}_{(2)}$) be the space of automorphic forms (resp. cusp forms, square-integrable forms) on $GU_{I}(\Q)\backslash GU_{I}(\AQ)$.  \\

We have inclusions for all $q$:
\begin{eqnarray}\nonumber   H^{q}(\lieG,K_{I,\infty};\mathcal{A}_{0}\otimes V_{\lambda})\subset \bar{H}^{q}(Sh_{I},V_{\lambda}^{\triangledown})\subset H^{q}(\lieG,K_{I,\infty};\mathcal{A}_{(2)}\otimes V_{\lambda})\\
\nonumber   H^{q}(\mathfrak{P},K_{I,\infty};\mathcal{A}_{0}\otimes V_{\Lambda})\subset \bar{H}^{q}(Sh_{I},E_{\Lambda})\subset H^{q}(\mathfrak{P},K_{I,\infty};\mathcal{A}_{(2)}\otimes V_{\Lambda}).
\end{eqnarray}

The complex conjugation on the automorphic forms induces a $K$-antilinear isomorphism:
\begin{equation} \label{definition cB}
c_{B}: \bar{H}^{0}(Sh_{I},E_{w_{0}*\lambda}) \xrightarrow{\sim}   \bar{H}^{D/2}(Sh_{I},E_{\lambda^{\vee}})
\end{equation}

More precisely, we summarize the construction in \cite{harris97} as follows.

\paragraph{Automorphic vector bundles:}
\text{}

We recall some facts on automorphic vector bundles first. We refer to page $113$ of \cite{harris97} and \cite{harrisarithmeticvectorbundle1} for notation and further details.

Let $(G,X)$ be a Shimura datum such that its special points are all CM points. Let $\widecheck{X}$ be the compact dual symmetric space of $X$. There is a surjective functor from the category of $G$-homogeneous vector bundles on $\widecheck{X}$ to the category of automorphic vector bundles on $Sh(G,X)$. This functor is compatible with inclusions of Shimura data as explained in the second part of Theorem $4.8$ of \cite{harrisarithmeticvectorbundle1}. It is also rational over the reflex field $E(G,X)$. \\

Let $\mathcal{E}$ be an automorphic vector bundle on $Sh(G,X)$ corresponding to $\mathcal{E}_{0}$, a $G$-homogeneous vector bundle on $\widecheck{X}$. Let $(T,x)$ be a special pair of $(G,X)$, i.e. $(T,x)$ is a sub-Shimura datum of $(G,X)$ with $T$ a maximal torus defined over $\Q$. Since the functor mentioned above is compatible with inclusions of Shimura datum, we know that the restriction of $\mathcal{E}$ to $Sh(T,x)$ corresponds to the restriction of $\mathcal{E}_{0}$ to $\check{x}\in \widecheck{X}$ by the previous functor. Moreover, by the construction, the fiber of $\mathcal{E}\mid_{Sh(T,x)}$ at any point of $Sh(T,x)$ is identified with the fiber of $\mathcal{E}_{0}$ at $\check{x}$. The $E(\mathcal{E})\cdot E(T,x)$-rational structure on the fiber of $\mathcal{E}_{0}$ at $\check{x}$ then defines a rational structure of $\mathcal{E}\mid_{Sh(T,x)}$ and called the \textbf{canonical trivialization} of $\mathcal{E}$ associated to $(T,x)$.

\paragraph{Complex conjugation on automorphic vector bundles:}
\text{}

Let $\mathcal{E}$ be as in page $112$ of \cite{harris97} and $\overline{\mathcal{E}}$ be its complex conjugation. The key step of the construction is to identify $\overline{\mathcal{E}}$ with the dual of $\mathcal{E}$ in a rational way. \\

More precisely, we recall Proposition $2.5.8$ of the \textit{loc.cit} that there exists a non-degenerate $G(\AQf)$-equivariant paring of real-analytic vector bundle $\mathcal{E}\otimes \overline{\mathcal{E}}\rightarrow \mathcal{E}_{\nu}$ such that its pullback to any CM point is rational with respect to the canonical trivializations. \\

We now explain the notion $\mathcal{E}_{\nu}$. Let $h\in X$ and $K_{h}$ be the stabilizer of $h$ in $G(\R)$. We know $\mathcal{E}$ is associated to an irreducible complex representation of $K_{h}$, denoted by $\tau$ in the \textit{loc.cit}. The complex conjugation of $\tau$ can be extended as an algrebraic representation of $K_{h}$, denoted by $\tau'$. We know $\tau'$ is isomorphic to the dual of $\tau$ and then there exists $\nu$, a one-dimensional representation $K_{h}$, such that a $K_{h}$-equivariant rational paring $V_{\tau}\otimes V_{\tau'}\rightarrow V_{\nu}$ exists. We denote by $\mathcal{E}_{\nu}$
the automorphic vector bundle associated to $V_{\nu}$.

In our case, we have $(G,X)=(GU_{I},X_{I})$, $h=h_{I}$ and $K_{h}=K_{I,\infty}$. Let $\tau=\Lambda=w_{0}*\lambda$ and $\mathcal{E}=E_{\Lambda}$. As explained in the last second paragraph before Corollary $2.5.9$ in the \textit{loc.cit}, we may identify the holomorphic sections of $V_{\Lambda}$ with holomorphic sections of the dual of $\overline{V_{\Lambda}}$. The complex conjugation then sends the latter to the anti-holomorphic sections of $\widecheck{V_{\Lambda}}=V_{\Lambda^{*}}$. The latter can be identified with harmonic (0,d)-forms  with values in $\mathbb{K}\otimes E_{\Lambda^{\*}}$ where $\mathbb{K}=\Omega^{D/2}_{Sh_{I}}$ is the canonical line bundle of $Sh_{I}$. \\

By $2.2.9$ of \cite{harris97} we have $\mathbb{K}=E_{(0,(-s_{\sigma},\cdots,-s_{\sigma},r_{\sigma},\cdots,r_{\sigma})_{\sigma\in\Sigma})}$ where the number of $-s_{\sigma}$ in the last term is $r_{\sigma}$. Therefore, complex conjugation gives an isomorphism:
\begin{equation}
c_{B}: \bar{H}^{0}(Sh_{I},E_{\Lambda}) \xrightarrow{\sim}   \bar{H}^{D/2}(Sh_{I},E_{\Lambda^{*}+0,(-s_{\sigma},\cdots,-s_{\sigma},r_{\sigma},\cdots,r_{\sigma})_{\sigma\in\Sigma})}).
\end{equation}

Recall equation (\ref{twist by w0}) that 
\begin{equation}\nonumber\Lambda=w_{0}*\lambda=(\lambda_{0},(\lambda_{s_{\sigma}+1}(\sigma)-s_{\sigma},\cdots,\lambda_{n}(\sigma)-s_{\sigma};\lambda_{1}(\sigma)+r_{\sigma},\cdots,\lambda_{s_{\sigma}}(\sigma)+r_{\sigma})_{\sigma\in\Sigma}).
\end{equation} 

We have \begin{equation}
\Lambda^{*}=(-\lambda_{0},(-\lambda_{n}(\sigma)+s_{\sigma},\cdots,-\lambda_{s_{\sigma}+1}(\sigma)+s_{\sigma};-\lambda_{s_{\sigma}}(\sigma)-r_{\sigma},\cdots,-\lambda_{1}(\sigma)+r_{\sigma})_{\sigma\in\Sigma}).
\end{equation}

Therefore, $\Lambda^{*}+(0,(-s_{\sigma},\cdots,-s_{\sigma},r_{\sigma},\cdots,r_{\sigma})_{\sigma\in\Sigma})=\lambda^{\vee}$. We finally get equation (\ref{definition cB}).

Similarly, if we start from the anti-holomorphic part, we will get a $K$-antilinear isomorphism which is still denoted by $c_{B}$:
\begin{equation} 
c_{B}: \bar{H}^{D/2}(Sh_{I},E_{\lambda}) \xrightarrow{\sim}   \bar{H}^{0}(Sh_{I},E_{w_{0}*\lambda^{\vee}})
\end{equation} which sends anti-holomorphic elements with respect to $\lambda$ to holomorphic elements for $\lambda^{\vee}$.

\subsection{The rational paring}
Let $\Lambda\in \Lambda(K_{I,\infty})$. We write $V=V_{\Lambda}$ in this section for simplicity.  As in section $2.6.11$ of \cite{harris97}, we denote by $\C_{\Lambda}$ the corresponding highest weight space. We know $\Lambda^{*}:=\Lambda^{\#}-(2\Lambda_{0},(0))$ is the tuple associated to $\widecheck{V}$, the dual of this $K_{I}$ representation. We denote by $\C_{-\Lambda}$ the lowest weight of $\widecheck{V}$.

The restriction from $V$ to $\C_{\Lambda}$ gives an isomorphism
\begin{equation}
Hom_{K_{I,\infty}}(V,\mathcal{C}^{\infty}(GU_{I}(\F)\backslash GU_{I}(\AF))) \xrightarrow{\sim} Hom_{H}(\C_{\Lambda},\mathcal{C}^{\infty}(GU_{I}(\F)\backslash GU_{I}(\AF))_{V}) 
\end{equation}
where $\mathcal{C}^{\infty}(GU_{I}(\F)\backslash GU_{I}(\AF))_{V}$ is the $V$-isotypic subspace of $\mathcal{C}^{\infty}(GU_{I}(\F)\backslash GU_{I}(\AF))$.

Similarly, we have \begin{equation}\label{trivialization}
Hom_{K_{I,\infty}}(\widecheck{V},\mathcal{C}^{\infty}(GU_{I}(\F)\backslash GU_{I}(\AF))) \xrightarrow{\sim} Hom_{H}(\C_{-\Lambda},\mathcal{C}^{\infty}(GU_{I}(\F)\backslash GU_{I}(\AF))_{\widecheck{V}}) 
\end{equation}

Proposition $2.6.12$ of \cite{harris97} says that up to a rational factor the perfect paring
\begin{equation}
Hom_{H}(\C_{\Lambda},\mathcal{C}^{\infty}(GU_{I}(\F)\backslash GU_{I}(\AF))_{V}) \times Hom_{H}(\C_{-\Lambda},\mathcal{C}^{\infty}(GU_{I}(\F)\backslash GU_{I}(\AF))_{\widecheck{V}}) 
\end{equation}
given by integration over the diagonal equals to restriction of the canonical paring (c.f. $(2.6.11.4)$ of \cite{harris97})
 \begin{eqnarray}
 &&Hom_{K_{I,\infty}}(V,\mathcal{C}^{\infty}(GU_{I}(\F)\backslash GU_{I}(\AF))) \times Hom_{K_{I,\infty}}(\widecheck{V},\mathcal{C}^{\infty}(GU_{I}(\F)\backslash GU_{I}(\AF))) 
\nonumber\\
&\rightarrow& Hom_{K_{I,\infty}}(V\otimes \widecheck{V},\mathcal{C}^{\infty}(GU_{I}(\F)\backslash GU_{I}(\AF)))\nonumber\\
&\rightarrow& Hom_{K_{I,\infty}}(\C,\mathcal{C}^{\infty}(GU_{I}(\F)\backslash GU_{I}(\AF)))\nonumber\\
&\rightarrow& \C.
 \end{eqnarray}

 We identify $\Gamma^{\infty}(Sh_{I},E_{\Lambda})$ with $Hom_{GU_{I}K_{I,\infty}}(\widecheck{V},\mathcal{C}^{\infty}(GU_{I}(\F)\backslash GU_{I}(\AF)))$ and  regard the latter as subspace of $Hom_{K_{I,\infty}}(\widecheck{V},\mathcal{C}^{\infty}(GU_{I}(\F)\backslash GU_{I}(\AF)))$.
 
 The above construction gives a $K$-rational perfect paring between holomorphic sections of $E_{\Lambda}$ and anti-holomorphic sections of $E_{\Lambda^{*}}$.
  
 If $\Lambda=w_{0}*\lambda$, as we have seen in Section \ref{Complex conjugation} that the anti-holomorphic sections of $E_{\Lambda^{*}}$ can be identified with harmonic $(0,d)$-forms with values in $E_{\lambda^{\vee}}$.
 
 We therefore obtain a $K$-rational perfect paring 
 \begin{equation}
 \Phi=\Phi^{I,\lambda}: \bar{H}^{0}(Sh_{I},E_{w_{0}*\lambda}) \times \bar{H}^{D/2}(Sh_{I},E_{\lambda^{\vee}})\rightarrow \C. 
 \end{equation}
 In other words, there is a rational paring between the holomorphic elements for $(I,\lambda)$ and anti-holomorphic elements for $(I,\lambda^{\vee})$.

It is easy to see that the isomorphism $Sh_{I}\xrightarrow{\sim} Sh_{I^{c}}$ commutes with the above paring and hence:

 \begin{lem}\label{period stable under complex conjugation}
For any $f\in  \bar{H}^{0}(Sh_{I},E_{w_{0}*\lambda}) $ and $g\in \bar{H}^{D/2}(Sh_{I},E_{\lambda^{\vee}})$, we have $\Phi^{I,\lambda}(f,g)=\Phi^{I^{c},\lambda^{c}}(c_{DR}f,c_{DR}g)$.
 \end{lem}
  
  The next lemma follows from Corollary $2.5.9$ and Lemma $2.8.8$ of \cite{harris97}.
  
 \begin{lem}
 Let $0\neq f\in \bar{H}^{0}(Sh_{I},E_{w_{0}*\lambda})$. We have
 $\Phi(f,c_{B}f)\neq 0$.
 
 More precisely, if we consider $f$ as an element in $$Hom_{K_{I,\infty}}(\widecheck{V},\mathcal{C}^{\infty}(GU_{I}(\F)\backslash GU_{I}(\AF)))$$ then by (\ref{trivialization}) and the fixed trivialization of $\C_{-w_{0}*\lambda}$, we may consider $f$ as an element in $\mathcal{C}^{\infty}(GU_{I}(\F)\backslash GU_{I}(\AF)))$. We have:
 \begin{equation}
 \Phi(f,c_{B}f) =\pm i^{\lambda_{0}}\int\limits_{GU_{I}(\Q)Z_{GU_{I}}(\AQ)\backslash GU_{I}(\AQ)} f(g)\overline{f}(g)||\nu(g)||^{c}dg.
 \end{equation}
 Recall that $\nu(\cdot)$ is the similitude defined in (\ref{definition of nu(g)}).
 \end{lem}

Similarly, if we start from anti-holomorphic elements, we get a paring:
\begin{equation}\Phi^{-}=\Phi^{I,\lambda,-}: \bar{H}^{D/2}(Sh_{I},E_{\lambda}) \times \bar{H}^{0}(Sh_{I},E_{w_{0}*\lambda^{\vee}})\rightarrow \C. \end{equation}
We use the script $\text{}^{-}$ to indicate that is anti-holomorphic. It is still $c_{DR}$ stable. For $0\neq f^{-} \in \bar{H}^{D/2}(Sh_{I},E_{\lambda}) $, we also know that $\Phi^{-}(f^{-},c_{B}f^{-})\neq 0$.

\subsection{Arithmetic automorphic periods}

Let $\pi$ be an irreducible cuspidal representation of $GU_{I}(\AQ)$ defined over a number field $E(\pi)$. We may assume that $E(\pi)$ contains the quadratic imaginary field $K$.

We assume that $\pi$ is cohomological with type $\lambda$, i.e. $H^{*}(\lieG,K_{I,\infty};\pi\otimes W_{\lambda})\neq 0$.

For $M$ a $GU_{I}(\AQf)$-module, define the $K$-rational $\pi_{f}$-isotypic components of $M$ by
\begin{equation}\nonumber
M^{\pi}:=Hom_{GU_{I}(\AFf)}(Res_{E(\pi)/K}(\pi_{f}),M)=\bigoplus\limits_{\tau \in \Sigma_{E(\pi)}} Hom(\pi_{f}^{\tau},M).
\end{equation} 

Therefore, if $M$ has a $K$-rational structure then $M^{\pi}$ also has a $K$-rational structure.

As in section \ref{Complex conjugation}, we have inclusions:
\begin{equation}\nonumber    H^{q}(\mathfrak{P},K_{I,\infty};\mathcal{A}_{0}^{\pi}\otimes V_{\Lambda})\subset \bar{H}^{q}(Sh_{I},E_{\Lambda})^{\pi}\subset H^{q}(\mathfrak{P},K_{I,\infty};\mathcal{A}_{(2)}^{\pi}\otimes V_{\Lambda}).\end{equation}

Under these inclusions, $c_{B}$ sends $\bar{H}^{0}(Sh_{I},E_{w_{0}*\lambda})^{\pi}$ to $\bar{H}^{D/2}(Sh_{I},E_{\lambda^{\vee}})^{\pi^{\vee}}$.

These inclusions are compatible with those $K$-rational structures and then induce $K$-rational parings
\begin{eqnarray}\Phi^{\pi}: \bar{H}^{0}(Sh_{I},E_{w_{0}*\lambda})^{\pi} \times \bar{H}^{D/2}(Sh_{I},E_{\lambda^{\vee}})^{\pi^{\vee}}\rightarrow \C\\
 \text{and }
\Phi^{-,\pi}: \bar{H}^{D/2}(Sh_{I},E_{\lambda})^{\pi} \times \bar{H}^{0}(Sh_{I},E_{w_{0}*\lambda^{\vee}})^{\pi^{\vee}}\rightarrow \C. \end{eqnarray}

\begin{df}
Let $\beta$ be a non zero $K$-rational element of $\bar{H}^{0}(Sh_{I},E_{w_{0}*\lambda})^{\pi}$. We define the \textbf{holomorphic arithmetic automorphic period associated to } $\beta$ by $P^{(I)}(\beta,\pi):=(\Phi^{\pi(}\beta^{\tau},c_{B}\beta^{\tau}))_{\tau\in\Sigma_{E(\pi)}}$. It is an element in $(E(\pi)\otimes_{K}\C)^{\times}$.

Let $\gamma$ be a non zero $K$-rational element of $\bar{H}^{D/2}(Sh_{I},E_{\lambda})^{\pi}$. We define the \textbf{anti-holomorphic arithmetic automorphic period associated to} $\gamma$ by $P^{(I),-}(\gamma,\pi):=(\Phi^{-,\pi}(\gamma^{\tau},c_{B}\gamma^{\tau}))_{\tau\in\Sigma_{E(\pi)}}$. It is an element in $(E(\pi)\otimes_{K}\C)^{\times}$.

\end{df}

\begin{dflem}
Let us assume now $\pi$ is tempered and $\pi_{\infty}$ is discrete series representation. In this case, $\bar{H}^{0}(Sh_{I},E_{w_{0}*\lambda})^{\pi}$ is a rank one $E(\pi)\otimes_{K} \C$-module (c.f. \cite{endoscopicfour}). \\

We define the \textbf{holomorphic arithmetic automorphic period} of $\pi$ by $P^{(I)}(\pi):= P^{(I)}(\beta,\pi)$ by taking $\beta$ any non zero rational element in $\bar{H}^{0}(Sh_{I},E_{w_{0}*\lambda})^{\pi}$. It is an element in $(E(\pi)\otimes_{K}\C)^{\times}$ well defined up to $E(\pi)^{\times}$-multiplication.

We define $P^{(I),-}(\pi)$ the \textbf{anti-holomorphic arithmetic automorphic period} of $\pi$ similarly.
\end{dflem}

\begin{lem}\label{pair to ratio}
We assume that $\pi$ is tempered and $\pi_{\infty}$ is discrete series representation. Let $\beta$ be a non zero rational element in $\bar{H}^{0}(Sh_{I},E_{w_{0}*\lambda})^{\pi}$ and $\beta^{\vee}$ be a non zero rational element in $\bar{H}^{0}(Sh_{I},E_{\lambda}^{\vee})^{\pi^{\vee}}$. 

We have $c_{B}(\beta)\sim_{E(\pi)} P^{(I)}(\pi) \beta^{\vee}$.
\end{lem}
\begin{dem}
It is enough to notice that $\Phi^{\pi}(\beta,\beta^{\vee})\in E(\pi)^{\times}$.
\end{dem}

\begin{lem}\label{lemma inverse}
If $\pi$ is tempered and $\pi_{\infty}$ is discrete series representation then we have:
\begin{enumerate}
\item $P^{(I^{c})}(\pi^{c})\sim_{E(\pi)} P^{(I)}(\pi)$.
\item $P^{(I)}(\pi^{\vee})*P^{(I),-}(\pi)\sim_{E(\pi)}1$.
\end{enumerate}
\end{lem}
\begin{dem}
The first part comes from Lemma \ref{period stable under complex conjugation} and the fact that $c_{DR}$ preserves rational structures. 

For the second part, recall that the following two parings are actually the same:
\begin{eqnarray}\Phi^{\pi^{\vee}}: \bar{H}^{0}(Sh_{I},E_{w_{0}*\lambda^{\vee}})^{\pi^{\vee}} \times \bar{H}^{D/2}(Sh_{I},E_{\lambda})^{\pi}\rightarrow \C\\
 \text{and }
\Phi^{-,\pi}: \bar{H}^{D/2}(Sh_{I},E_{\lambda})^{\pi} \times \bar{H}^{0}(Sh_{I},E_{w_{0}*\lambda^{\vee}})^{\pi^{\vee}}\rightarrow \C. \end{eqnarray}

We take $\beta$ a rational element in $\bar{H}^{0}(Sh_{I},E_{w_{0}*\lambda^{\vee}})^{\pi^{\vee}}$ and $\gamma$ a rational element in $\bar{H}^{D/2}(Sh_{I},E_{\lambda})^{\pi}$. We may assume that $\Phi^{\pi^{\vee}}(\beta^{\tau},\gamma^{\tau})=\Phi^{-,\pi}(\gamma^{\tau},\beta^{\tau})=1$ for all $\tau\in \Sigma_{E(\pi)}$.

By definition $p^{(I)}(\pi^{\vee})=(\Phi^{\pi^{\vee}}(\beta^{\tau},c_{B}\beta^{\tau}))_{\tau\in\Sigma_{E(\pi)}}$. Since $\bar{H}^{D/2}(Sh_{I},E_{\lambda})^{\pi}$ is a rank one $E(\pi)\otimes \C$-module, there exists $C \in (E(\pi)\otimes_{K} \C)^{\times}$ such that $(c_{B}\beta^{\tau})_{\tau\in\Sigma_{E(\pi)}}=C (\gamma^{\tau})_{\tau\in\Sigma_{E(\pi)}}$. Therefore $p^{I}(\pi^{\vee})=C (\Phi^{\pi^{\vee}}(\beta^{\tau},\gamma^{\tau}))_{\tau\in\Sigma_{E(\pi)}}=C$.

On the other hand, since $c_{B}^{2}=Id$, we have $(c_{B}\gamma^{\tau})_{\tau\in\Sigma_{E(\pi)}}=C^{-1}(\beta^{\tau})_{\tau\in\Sigma_{E(\pi)}}$. We can deduce that $p^{(I),-}(\pi)=C^{-1}$ as expected.
\end{dem}
\begin{df}
We say $I$ is \textbf{compact} if $U_{I}(\C)$ is. In other words, $I$ is compact if and only if $I(\sigma)=0$ or $n$ for all $\sigma\in\Sigma$.
\end{df}

\begin{cor}\label{unitary similitude product 1}
If $I$ is compact then $P^{(I)}(\pi) \sim_{E(\pi)}P^{(I),-}(\pi)$.
We have $P^{(I)}(\pi^{\vee})*P^{(I)}(\pi)\sim_{E(\pi)}1$.
\end{cor}
\begin{dem}
If $I$ is compact, then $w_{0}=Id$. The anti-holomorphic part and holomorphic part are the same. We then have $P^{(I)}(\pi) \sim_{E(\pi)}P^{(I),-}(\pi)$. The last assertion comes from Lemma \ref{lemma inverse}.
\end{dem}

 The following theorem is Theorem $4.3.3$ of \cite{guerberofflin} which generalizes the main theorem of \cite{guerbperiods} and \cite{harris97}:
 \begin{thm}\label{main theorem CM}\label{n*1}
 Let $\Pi$ be a regular, conjugate self-dual, cohomological, cuspidal automorphic representation of $GL_{n}(\AF)$ which descends to $U_{I}(\AFp)$ for any $I$. We denote the infinity type of $\Pi$ at $\sigma\in \Sigma$ by $(z^{a_{i}(\sigma)}\overline{z}^{-a_{i}(\sigma)})_{1\leq i\leq n}$. 

 Let $\eta$ be an algebraic Hecke character of $F$ with infinity type $z^{a(\sigma)}\overline{z}^{b(\sigma)}$ at $\sigma\in \Sigma$. We know that $a(\sigma)+b(\sigma)$ is a constant independent of $\sigma$, denoted by $-\omega(\eta)$.
 
 We suppose that $a(\sigma)-b(\sigma)+2a_{i}(\sigma)\neq 0$ for all $1\leq i\leq n$ and $\sigma\in \Sigma$. We define $I:=I(\Pi,\eta)$ to be the map on $\Sigma$ which sends $\sigma\in\Sigma$ to $I(\sigma):=\#\{i:a(\sigma)-b(\sigma)+2a_{i}(\sigma)<0\}$. 
 
Let $m\in \Z+\cfrac{n-1}{2}$. If $m\geq \cfrac{1+\omega(\eta)}{2}$ is critical for $\Pi\otimes \eta$, we have:

\begin{equation}\label{main result CM}
L(m,\Pi\otimes \eta) \sim_{E(\Pi)E(\eta)} (2 \pi i)^{mnd} P^{(I(\Pi,\eta))}(\Pi) \prod\limits_{\sigma\in\Sigma}p(\widecheck{\eta},\sigma)^{I(\sigma)}p(\widecheck{\eta},\overline{\sigma})^{n-I(\sigma)}.
\end{equation}
and is equivariant under the action of $F^{Gal}$.
Here $E(\Pi)$ is the compositum of all $E(\pi)$ when $I$ varies among all the signatures.
 \end{thm}

 The aim of this paper is to prove the following conjecture which generalizes a conjecture of Shimura (\cite{shimura83}):
 \begin{conj}\label{factorization conjecture}
 There exists some non zero complex numbers $P^{(s)}(\Pi,\sigma)$ for all $0\leq s\leq n$ and $\sigma\in\Sigma$ such that $P^{(I)}(\Pi) \sim_{E(\Pi)} \prod\limits_{\sigma\in\Sigma}P^{(I(\sigma))}(\Pi,\sigma)$ for all $I=(I(\sigma))_{\sigma\in\Sigma}\in \{0,1,\cdots,n\}^{\Sigma}$.
  \end{conj}

\section{Factorization of arithmetic automorphic periods and a conjecture}

\subsection{Basic lemmas}
 Let $X$, $Y$ be two sets and $Z$ be a multiplicative abelian group. We will apply the result of this section to $Z=\C^{\times}/E^{\times}$ where $E$ is a proper number field.
 
\begin{lem}
 Let $f$ be a map from $X\times Y$ to $Z$. The following two statements are equivalent:
 \begin{enumerate}
 \item There exists two maps $g:X\rightarrow Z$ and $h:Y\rightarrow Z$ such that $f(x,y)=g(x)h(y)$ for all $(x,y)\in X\times Y$.
 \item For all $x,x'\in X$ and $y,y'\in Y$, we have $f(x,y)f(x',y')=f(x,y')f(x',y)$.
 \end{enumerate} 
 Moreover, if the above equivalent statements are satisfied, the maps $g$ and $h$ are unique up to scalars.
 
 \end{lem}
 
 \begin{dem}
 The direction that $1$ implies $2$ is trivial. Let us prove the inverse. 
 We fix any $y_{0}\in Y$ and define $g(x):=f(x,y_{0})$ for all $x\in X$. We then fix any $x_{0}\in X$ and define $h(y):=\cfrac{f(x_{0},y)}{g(x_{0})}=\cfrac{f(x_{0},y)}{f(x_{0},y_{0})}$. 
 
For any $x\in X$ and $y\in Y$, Statement $2$ tells us that $f(x,y)f(x_{0},y_{0})=f(x,y_{0})f(x_{0},y)$. Therefore $f(x,y)=f(x,y_{0})\times\cfrac{f(x_{0},y)}{f(x_{0},y_{0})}=g(x)h(y)$ as expected.

  \end{dem}

 Let $n$ be a positive integer and $X_{1},\cdots,X_{n}$ be some sets. Let $f$ be a map from $X_{1}\times X_{2}\times \cdots\times X_{n}$ to $Z$.
 
 The following corollary can be deduced from the above Lemma by induction on $n$.
 
  \begin{cor}\label{factorization lemma}
 
The following two statements are equivalent:
 \begin{enumerate}
 \item There exists some maps $f_{k}:X_{k}\rightarrow Z$ for $1\leq k\leq n$ such that $f(x_{1},x_{2},\cdots,x_{n})=\prod\limits_{1\leq k\leq n} f_{k}(x_{k})$ for all $x_{k}\in X_{k}$, $1\leq k\leq n$.
 \item Given any $x_{j},x'_{j}\in X_{j}$ for each $1\leq j\leq n$, we have 
\begin{eqnarray}
&f(x_{1},x_{2},\cdots,x_{n})\times f(x'_{1},x'_{2},\cdots,x'_{n})&\nonumber\\
&=f(x_{1},\cdots,x_{k-1},x'_{k},x_{k+1},x_{n})\times f(x'_{1},\cdots x'_{k-1},x_{k},x'_{k+1},\cdots,x'_{n})&\nonumber
\end{eqnarray} for any $1\leq k\leq n$.
 \end{enumerate} 
 Moreover, if the above equivalent statements are satisfied then for any $\lambda_{1},\cdots,\lambda_{n}\in Z$ such that $\lambda_{1}\cdots\lambda_{n}=1$, we have another factorization $f(x_{1},\cdots,x_{n})=\prod\limits_{1\leq i\leq n}(\lambda_{i}f_{i})(x_{i})$. Each factorization of $f$ is of the above form.

We fix $a_{i}\in X_{i}$ for each i and $c_{1},\cdots,c_{n}\in Z$ such that $f(a_{1},\cdots,a_{n})=c_{1}\cdots c_{n}$. If the above equivalent statements are satisfied then there exists a unique factorization such that $f_{i}(a_{i})=c_{i}$.

  \end{cor}

\begin{rem}  \label{remark factorization}
If $\#X_{k}\geq 3$ for all $k$, it is enough to verify the condition in statement $2$ of the above corollary in the case $x_{j}\neq x'_{j}$ for all $1\leq j\leq n$.

In fact, when $\#X_{k}\geq 3$ for all $k$, for any $1\leq j\leq n$ and any $y_{j},y'_{j}\in X_{j}$, we may take $x_{j} \in X_{j}$ such that $x_{j}\neq y_{j}$, $x_{j}\neq y_{j}'$.

We fix any $1\leq k\leq n$. If statement $2$ is verified when $x_{j}\neq x'_{j}$ for all $j$ then for any $y_{k}\neq y'_{k}$, we have
\begin{eqnarray}
&&f(y_{1},y_{2},\cdots,y_{n})f(y'_{1},y'_{2},\cdots,y'_{n}) f(x_{1},x_{2},\cdots,x_{n}) \nonumber\\
&=&f(y_{1},y_{2},\cdots,y_{n})f(y'_{1},\cdots y'_{k-1},x_{k},y'_{k+1},\cdots,y'_{n})\times\nonumber\\
&& \hspace{20pt} f(x_{1},\cdots,x_{k-1},y'_{k},x_{k+1},\cdots x_{n}) \nonumber\\
 &=&f(y_{1},y_{2},\cdots,y_{n})f(x_{1},\cdots,x_{k-1},y'_{k},x_{k+1},\cdots x_{n})\times\nonumber\\
&&  \hspace{20pt} f(y'_{1},\cdots y'_{k-1},x_{k},y'_{k+1},\cdots,y'_{n})
 \nonumber\\
 &=& f(y_{1},\cdots,y_{k-1},y'_{k},y_{k+1},\cdots,y_{n})f(x_{1},\cdots,x_{k-1},y_{k},x_{k+1},\cdots,x_{n})\times \nonumber\\
 && \hspace{20pt} f(y'_{1},\cdots y'_{k-1},x_{k},y'_{k+1},\cdots,y'_{n})\nonumber\\
 &=& f(y_{1},\cdots,y_{k-1},y'_{k},y_{k+1},\cdots,y_{n})f(y'_{1},\cdots y'_{k-1},y_{k},y'_{k+1},\cdots,y'_{n})f(x_{1},x_{2},\cdots,x_{n}).\nonumber
\end{eqnarray}
We have assumed $y_{k}\neq y'_{k}$ to guarantee that each time we apply the formula in Statement $2$, the coefficients satisfy $x_{j}\neq x'_{j}$ for all $1\leq j\leq n$.

Therefore \begin{eqnarray}
&f(y_{1},y_{2},\cdots,y_{n})\times f(y'_{1},y'_{2},\cdots,y'_{n})& \nonumber\\
&=f(y_{1},\cdots,y_{k-1},y'_{k},y_{k+1},\cdots,y_{n})\times f(y'_{1},\cdots y'_{k-1},y_{k},y'_{k+1},\cdots,y'_{n})&\nonumber
\end{eqnarray} if $y_{k}\neq y'_{k}$. If $y_{k}=y'_{k}$, this formula is trivially true. \\

We conclude that we can weaken the condition in Statement $2$ of the above Corollary to $x_{j}\neq x'_{j}$ for all $1\leq j\leq n$ when $\#X_{k}\geq 3$ for all $k$. We will verify this weaker condition in the application to the factorization of arithmetic automorphic periods.
\end{rem}
  
\subsection{Relation of the Whittaker period and arithmetic periods}
 Let $\Pi$ be a regular cuspidal representation of $GL_{n}(\AF)$ as in Theorem \ref{main theorem CM} with infinity type $(z^{a_{i}(\sigma)}\overline{z}^{-a_{i}(\sigma)})_{1\leq i\leq n}$ at $\sigma\in \Sigma$. We may assume that $a_{1}(\sigma)>a_{2}(\sigma)>\cdots>a_{n}(\sigma)$ for all $\sigma\in\Sigma$.
 
Recall that we say $\Pi$ is \textbf{$N$-regular} if $a_{i}(\sigma)-a_{i+1}(\sigma)\geq N$ for all $1\leq i\leq n-1$ and $\sigma\in\Sigma$.

\begin{thm}\label{relation Whittaker}\label{main step for factorization}
For $1\leq i\leq n-1$, let $I_{u}$ be a map from $\Sigma$ to $\{1,\cdots,n-1\}$.
There exists a non-zero complex number $Z(\Pi_\infty)$ depending only on the infinity type of $\Pi$, such that if for any $\sigma\in\Sigma$, each number inside $\{1,2,\cdots,n-1\}$ appears exactly once in $\{I_{u}(\sigma)\}_{1\leq i\leq n-1}$, then we have:
 \begin{equation}\label{Whittaker period whole formula}
 p(\Pi)\sim_{E(\Pi)E(\Pi^{\#})} Z(\Pi_{\infty})\prod\limits_{1\leq u\leq n-1} P^{(I_{u})}(\Pi)
  \end{equation}
  provided $\Pi$ is 3-regular or certain central $L$-values are non-zero. 
\end{thm}

\begin{proof}

Let us assume at first that $n$ is even.

For each $\sigma$ and $u$, let $k_{u}(\sigma)$ be an integer such that $I_{u}(\sigma)=\#\{i\mid-a_{i}(\sigma)> k_{u}(\sigma)\}$. 
 
Since $n$ is even, $a_{i}(\sigma)\in \Z+\cfrac{1}{2}$ for all $1\leq i\leq n$ and all $\sigma\in \Sigma$. The condition on $I_{u}$ implies that for all $\sigma\in\Sigma$, the numbers $\{k_{u}(\sigma)\mid 1\leq u\leq n-1\}$ lie in the $n-1$ gaps between $-a_{n}(\sigma)>-a_{n-1}(\sigma)>\cdots>-a_{1}(\sigma)$.

For $1\leq u \leq n-1$, let $\chi_{u}$ be an algebraic conjugate self-dual Hecke character of $F$ with infinity type $z^{k_{u}(\sigma)}\overline{z}^{-k_{u}(\sigma)}$ at $\sigma\in \Sigma$.

  We define $\Pi^{\#}$ to be the Langlands sum of $\chi_{u}$, $1\leq u\leq n-1$. It is an algebraic regular automorphic representation of $GL_{n-1}(\AF)$. Then the pair $(\Pi,\Pi^{\#})$ is in good position. By Proposition \ref{Whittaker period theorem CM} we have
 \begin{equation}\label{CM Whittaker period}
L(\cfrac{1}{2}+m,\Pi\times \Pi^{\#})\sim_{E(\Pi)E(\Pi^{\#})} p(\Pi)p(\Pi^{\#})p(m,\Pi_{\infty},\Pi^{\#}_{\infty})
\end{equation}
where $p(m,\Pi_{\infty},\Pi^{\#}_{\infty})$ is a complex number which depends on $m,\Pi_{\infty}$ and $\Pi^{\#}_{\infty}$.

 Since $\Pi^{\#}$ is the Langlands sum of $\chi_{u}$, $1\leq u\leq n-1$, we have \begin{equation}\nonumber L(\cfrac{1}{2}+m,\Pi\times \Pi^{\#}) =\prod\limits_{1\leq u\leq n-1}L(\cfrac{1}{2}+m,\Pi\times \chi_{u}).
 \end{equation}
  We then apply Theorem \ref{n*1} to the right hand side and get:
 \begin{eqnarray}
 &&L(\cfrac{1}{2}+m,\Pi\times \Pi^{\#}) =\prod\limits_{1\leq u\leq n-1}L(\cfrac{1}{2}+m,\Pi\times \chi_{u})\nonumber\\
 &\sim_{E(\Pi)E(\Pi^{\#})}&\prod\limits_{1\leq u\leq n-1}[(2\pi i)^{d(m+\frac{1}{2})n}P^{(I(\Pi,\chi_{u}))}(\Pi)\prod\limits_{\sigma\in\Sigma}p(\widecheck{\chi_{u}},\sigma)^{I_{u}(\sigma)} p(\widecheck{\chi_{u}},\overline{\sigma})^{n-I_{u}(\sigma)} ]\nonumber
 \end{eqnarray}
Recall that $I(\Pi,\chi_{u})(\sigma)=\#\{i\mid-a_{i}(\sigma)> k_{u}(\sigma)\}=I_u(\sigma)$ for any $\sigma\in\Sigma$ and $1\leq u\leq n-1$.
 
 Note that $\chi_{u}$ is conjugate self-dual, we have $p(\widecheck{\chi_{u}},\overline{\sigma})\sim_{E(\Pi^{\#})}p(\widecheck{\chi_{u}^{c}},\sigma)\sim_{E(\Pi^{\#})}p(\widecheck{\chi_{u}^{-1}},\sigma)\sim_{E(\Pi^{\#})}p(\widecheck{\chi_{u}},\sigma)^{-1}$. We deduce that:
 
 \begin{equation}\label{the left hand side step 1}
L(\cfrac{1}{2}+m,\Pi\times \Pi^{\#}) \sim_{E(\Pi)E(\Pi^{\#})} (2\pi i)^{d(m+\frac{1}{2})n(n-1)}\prod\limits_{1\leq u\leq n-1} [P^{(I_{u})}(\Pi)\prod\limits_{\sigma\in\Sigma}p(\widecheck{\chi_{u}},\sigma)^{2I_{u}(\sigma)-n}  ]\end{equation}

 By Thoerem {Whittaker period theorem CM}, there exists a constant $\Omega(\Pi^{\#}_{\infty})\in \C^{\times}$ well defined up to $E(\Pi^{\#})^{\times}$ such that 
 \begin{equation}
 p(\Pi^{\#})\sim_{E(\Pi^{\#})} \Omega(\Pi^{\#}_{\infty})\prod\limits_{1\leq u<v\leq n-1}L(1,\chi_{u}\chi_{v}^{-1}).
 \end{equation}
 
 By Blasius's result, we have:
 \begin{equation}\nonumber
 L(1,\chi_{u}\chi_{v}^{-1}) \sim_{E(\Pi^{\#})} (2\pi i)^{d}\prod\limits_{\sigma\in\Sigma} p(\widecheck{\chi_{u}\chi_{v}^{-1}},\sigma')
 \end{equation}
 
 where the embedding $\sigma'$ is defined as follows:
 if $k_{u}(\sigma)<k_{v}(\sigma)$ then $\sigma'=\sigma$ and $p(\widecheck{\chi_{u}\chi_{v}^{-1}},\sigma') \sim_{E(\chi_{u})} p(\widecheck{\chi_{u}},\sigma) p(\widecheck{\chi_{v}},\sigma)^{-1}$; otherwise $\sigma'=\overline{\sigma}$ and
$
p(\widecheck{\chi_{u}\chi_{v}^{-1}},\sigma')\sim_{E(\chi_{u})}  p(\widecheck{\chi_{u}},\sigma)^{-1}p(\widecheck{\chi_{v}},\sigma)
$.

 Therefore, the Whittaker period $p(\Pi^{\#})$
 \begin{equation}\label{the right hand side step 1}
  \sim_{E(\Pi^{\#})}(2\pi i)^{\frac{d(n-1)(n-2)}{2}}\Omega(\Pi^{\#}_{\infty}) \prod\limits_{1\leq u\leq n-1}\prod\limits_{\sigma\in \Sigma} p(\widecheck{\chi_{u}},\sigma)^{\#\{v\mid k_{v}(\sigma)>k_{u}(\sigma)\}-\#\{v\mid k_{v}(\sigma)<k_{u}(\sigma)\}}
 \end{equation}
 
 We know $\#\{v\mid k_{v}(\sigma)<k_{u}(\sigma)\}=n-2-\#\{v\mid k_{v}(\sigma)>k_{u}(\sigma)\}$.
 
 Moreover, by definition of $k_{u}(\sigma)$ we have $\#\{v\mid k_{v}(\sigma)>k_{u}(\sigma)\}=\#\{i\mid -a_{i}(\sigma)>k_{u}(\sigma)\}-1=I_{u}(\sigma)-1$. 
 Therefore, 
 \begin{equation}\label{the right hand side step 2}
 \#\{v\mid k_{v}(\sigma)>k_{u}(\sigma)\}-\#\{v\mid k_{v}(\sigma)<k_{u}(\sigma)\}=2I_{u}(\sigma)-n
 \end{equation}

 We compare equations (\ref{CM Whittaker period}), (\ref{the left hand side step 1}), (\ref{the right hand side step 1}) and (\ref{the right hand side step 2}). If $\Pi$ is $3$-regular we may take $m=1$ and then $L(\cfrac{1}{2}+m,\Pi\times \Pi^{\#})$ is automatically non-zero, otherwise we take $m=0$ and we assume that $L(\cfrac{1}{2},\Pi\times \Pi^{\#})\neq 0$. We obtain that:
 \begin{eqnarray}
&& (2\pi i)^{d(m+\frac{1}{2})n(n-1)}\prod\limits_{1\leq u\leq n-1} P^{(I_{u})}(\Pi)\nonumber\\
 &\sim_{E(\Pi)E(\Pi^{\#})}  &(2\pi i)^{\frac{d(n-1)(n-2)}{2}}p(\Pi)\Omega(\Pi^{\#}_{\infty}) p(m,\Pi_{\infty},\Pi^{\#}_{\infty}).\nonumber
 \end{eqnarray}

 Hence we have $p(\Pi)\sim_{E(\Pi)E(\Pi^{\#})}$
 \begin{equation}\nonumber
 (2\pi i)^{d(m+\frac{1}{2})n(n-1)-\frac{d(n-1)(n-2)}{2}}\Omega(\Pi^{\#}_{\infty})^{-1} p(m,\Pi_{\infty},\Pi^{\#}_{\infty})^{-1}\prod\limits_{1\leq u\leq n-1} P^{(I_{u})}(\Pi).
 \end{equation}
 
 If we take $$Z(m,\Pi_{\infty},\Pi'_{\infty}):=(2\pi i)^{d(m+\frac{1}{2})n(n-1)-\frac{d(n-1)(n-2)}{2}}\Omega(\Pi^{\#}_{\infty})^{-1} p(m,\Pi_{\infty},\Pi^{\#}_{\infty})^{-1}$$ then $p(\Pi)\sim_{E(\Pi)E(\Pi^{\#})} Z(m,\Pi_{\infty},\Pi'_{\infty})\prod\limits_{1\leq u\leq n-1} P^{(I_{u})}(\Pi).$ 
 
 In particular, we have that $Z(m,\Pi_{\infty},\Pi'_{\infty})$ depends only on $\Pi_{\infty}$. \\
 
 We may define:
 \begin{equation}\label{definition of Z}
 Z(\Pi_{\infty}):=Z(m,\Pi_{\infty},\Pi'_{\infty})=(2\pi i)^{d(m+\frac{1}{2})n(n-1)-\frac{d(n-1)(n-2)}{2}}\Omega(\Pi^{\#}_{\infty})^{-1} p(m,\Pi_{\infty},\Pi^{\#}_{\infty})^{-1}.
 \end{equation}
  It is a non-zero complex number well defined up to elements in $E(\Pi)^{\times}$.

 We deduce that:
 \begin{equation}
  p(\Pi)\sim_{E(\Pi)E(\Pi^{\#})} Z(\Pi_{\infty})\prod\limits_{1\leq u\leq n-1} P^{(I_{u})}(\Pi).
  \end{equation}

Now assume that $n$ is odd. We keep the notation in the above section We have $a_{i}(\sigma)\in \Z$ for all $1\leq i\leq n$ and all $\sigma\in \Sigma$. In this case, we take integers $k_{u}(\sigma)$ such that $I_{u}(\sigma)=\#\{i\mid-a_{i}(\sigma)> k_{u}(\sigma)+\cfrac{1}{2}\}$.

We still let $\chi_{u}$ be an algebraic conjugate self-dual Hecke character of $F$ with infinity type $z^{k_{u}(\sigma)}\overline{z}^{-k_{u}(\sigma)}$ at $\sigma\in \Sigma$.

Recall that $\psi$ is an algebraic Hecke character of $F$ with infinity type $z^{1}$ at each $\sigma\in \Sigma$ such that $\psi\psi^{c}=||\cdot||_{\AF}$. We take $\Pi^{\#}$ to be the Langlands sum of $\chi_{u}\psi||\cdot||_{\AF}^{-\frac{1}{2}}$, $1\leq u\leq n-1$. It is an algebraic regular automorphic representation of $GL_{n-1}(\AF)$. The conditions of Theorem \ref{Whittaker period theorem CM} hold. \\
 
 We repeat the above process for $\Pi$ and $\Pi^{\#}$ and get
  \begin{eqnarray}\nonumber
&L(\cfrac{1}{2}+m,\Pi\times \Pi^{\#})&\nonumber\\\nonumber
& \sim_{E(\Pi)E(\Pi^{\#})} (2\pi i)^{dmn(n-1)}\prod\limits_{1\leq u\leq n-1} [P^{(I(\Pi,\chi_{u}\psi))}\prod\limits_{\sigma\in\Sigma}p(\widecheck{\chi_{u}},\sigma)^{2I_{u}(\sigma)-n}  ]\times&\\\nonumber
&\prod\limits_{\sigma\in\Sigma}(p(\widecheck{\psi},\sigma)^{\sum_{1\leq u\leq n-1}I_{u}(\sigma)}p(\widecheck{\psi}^{c},\sigma)^{\sum_{1\leq u\leq n-1}(n-I_{u}(\sigma))})&
\end{eqnarray}
 where $I_{u}:=I(\Pi,\chi_{u}\psi)$ with $I_{u}(\sigma)=\#\{i\mid-a_{i}(\sigma)> k_{u}(\sigma)+\cfrac{1}{2}\}$.

 We see $\prod\limits_{1\leq u\leq n-1}I_{u}(\sigma)=\cfrac{n(n-1)}{2}$ and $\sum_{1\leq u\leq n-1}(n-I_{u}(\sigma))=\cfrac{n(n-1)}{2}$. \\
 
 We then have 
 \begin{eqnarray}
&& \prod\limits_{\sigma\in\Sigma}(p(\widecheck{\psi},\sigma)^{\sum_{1\leq u\leq n-1}I_{u}(\sigma)}p(\widecheck{\psi}^{c},\sigma)^{\sum_{1\leq u\leq n-1}(n-I_{u}(\sigma))}
)\nonumber\\
&\sim_{E(\psi)}& \prod\limits_{\sigma\in\Sigma}p(\widecheck{\psi\psi^{c}},\sigma)^{\frac{n(n-1)}{2}} \sim_{E(\psi)}  \prod\limits_{\sigma\in\Sigma}p(||\cdot||_{\AF}^{-1},\sigma)^{\frac{n(n-1)}{2}} \sim_{E(\psi)} (2\pi i)^{\frac{dn(n-1)}{2}}\nonumber.
 \end{eqnarray}
 
 We verify that the equation (\ref{the right hand side step 1}) and (\ref{the right hand side step 2}) remain unchanged. We can see that equation (\ref{Whittaker period whole formula}) still holds here.

\end{proof}

\subsection{Factorization of arithmetic automorphic periods: restricted case}

  We consider the function $\prod\limits_{\sigma\in\Sigma} \{0,1,\cdots,n\}\rightarrow \C^{\times}/E(\Pi)^{\times}$ which sends  $(I(\sigma))_{\sigma\in\Sigma}$ to $P^{(I)}(\Pi)$.

  In this section, we will prove the above conjecture restricted to $\{1,2,\cdots,{n-1}\}^{\Sigma}$. More precisely, we will prove that
  
 \begin{thm}\label{restricted theorem}
If $n\geq 4$ and $\Pi$ satisfies a global non vanishing condition, in particular, if $\Pi$ is $3$-regular, then there exists some non zero complex numbers $P^{(s)}(\Pi,\sigma)$ for all $1\leq s\leq n-1$, $\sigma\in\Sigma$ such that $P^{(I)}(\Pi) \sim_{E(\Pi)} \prod\limits_{\sigma\in\Sigma}P^{(I(\sigma))}(\Pi,\sigma)$ for all $I=(I(\sigma))_{\sigma\in\Sigma}\in \{1,2,\cdots,n-1\}^{\Sigma}$.

 \end{thm}
\begin{dem}
 For all $\sigma\in\Sigma$, let $I_{1}(\sigma)\neq I_{2}(\sigma)$ be two numbers in $\{1,2,\cdots,n-1\}$. We consider $I_{1},I_{2}$ as two elements in $\{1,2,\cdots,n-1\}^{\Sigma}$.
 
 Let $\sigma_{0}$ be any element in $\Sigma$. We define $I'_{1}, I'_{2}\in \{1,2,\cdots,n-1\}^{\Sigma}$ by $I_{1}'(\sigma):=I_{1}(\sigma)$, $I_{2}'(\sigma):=I_{2}(\sigma)$ if $\sigma\neq \sigma_{0}$ and $I'_{1}(\sigma_{0}):=I_{2}(\sigma_{0})$, $I'_{2}(\sigma_{0}):=I_{1}(\sigma_{0})$.
 
  By Remark \ref{remark factorization}, it is enough to prove that 
  \begin{equation}\nonumber
  P^{(I_{1})}(\Pi)P^{(I_{2})}(\Pi)\sim_{E(\Pi)}P^{(I'_{1})}(\Pi)P^{(I'_{2})}(\Pi).
  \end{equation}

Since $I_{1}(\sigma)\neq I_{2}(\sigma)$  for all $\sigma\in\Sigma$, we can always find $I_{3},\cdots,I_{n-1}\in \{1,2,\cdots,n-1\}^{\Sigma}$ such that for all $\sigma\in\Sigma$, the $(n-1)$ numbers $I_{u}(\sigma)$, $1\leq u\leq n-1$ run over $1,2,\cdots,n-1$. In other words, conditions in Theorem \ref{main step for factorization} are verified.
   
 By Theorem \ref{main step for factorization}, we have 
 \begin{equation}\nonumber p(\Pi)\sim_{E(\Pi)} Z(\Pi_{\infty})  P^{(I_{1})}(\Pi)P^{(I_{2})}(\Pi)\prod\limits_{3\leq u\leq n-1} P^{(I_{u})}(\Pi).
 \end{equation}
 
 On the other hand, it is easy to see that $I'_{1}$, $I'_{2}$, $I_{3},\cdots,I_{n-1}$ also satisfy conditions in Theorem \ref{main step for factorization}. Therefore  \begin{equation}\nonumber
 p(\Pi)\sim_{E(\Pi)} Z(\Pi_{\infty})  P^{(I'_{1})}(\Pi)P^{(I'_{2})}(\Pi)\prod\limits_{3\leq u\leq n-1} P^{(I_{u})}(\Pi).
\end{equation}

We conclude at last $P^{(I_{1})}(\Pi)P^{(I_{2})}(\Pi)\sim_{E(\Pi)}P^{(I'_{1})}(\Pi)P^{(I'_{2})}(\Pi)$ and then the above theorem follows.
 \end{dem}
 
  \begin{cor}\label{Whittaker period local formula}
If $\Pi$ satisfied the conditions in the above theorem then we have:
\begin{equation}
p(\Pi)
\sim_{E(\Pi)} Z(\Pi_{\infty}) \prod\limits_{\sigma\in\Sigma} \prod\limits_{1\leq i\leq n-1} P^{(i)}(\Pi,\sigma)
\end{equation}
\end{cor}

  \subsection{Factorization of arithmetic automorphic periods: complete case}\label{complete case}
In this section, we will prove Conjecture \ref{factorization conjecture} when $\Pi$ is regular enough. More precisely, we have

\begin{thm}\label{complete theorem}
Conjecture \ref{factorization conjecture} is true provided that $\Pi$ is $2$-regular and satisfies a global non vanishing condition which is automatically satisfied if $\Pi$ is $6$-regular.
\end{thm}

\begin{proof}

 If $n=1$, Conjecture \ref{factorization conjecture} is known as multiplicity of CM periods (see Proposition \ref{propCM}). We may assume that $n\geq 2$. The set $\{0,1,\cdots,n\}$ has at least $3$ elements and then Remark \ref{remark factorization} can apply.
 
  For all $\sigma\in\Sigma$, let $I_{1}(\sigma)\neq I_{2}(\sigma)$ be two numbers in $\{0,1,\cdots,n\}$. We have $I_{1},I_{2}\in \{0,1,2,\cdots,n\}^{\Sigma}$.
 
 Let $\sigma_{0}$ be any element in $\Sigma$. We define $I'_{1}, I'_{2}\in \{0,1,2,\cdots,n\}^{\Sigma}$ as in the proof of Theorem \ref{restricted theorem}. \\
 
 It remains to show that \begin{equation}\label{factorization equation}
 P^{(I_{1})}(\Pi)P^{(I_{2})}(\Pi)\sim_{E(\Pi)}P^{(I'_{1})}(\Pi)P^{(I'_{2})}(\Pi).
 \end{equation}\\
 
 Let us assume that $n$ is odd at first.
 Since $\Pi$ is $2$-regular, we can find $\chi_{u}$ a conjugate self-dual algebraic Hecke character of $F$ such that $I(\Pi,\chi_{u})=I_{u}$ for $u=1,2$. We denote the infinity type of $\chi_{u}$ at $\sigma\in\Sigma$ by $z^{k_{u}(\sigma)}\overline{z}^{-k_{u}(\sigma)}$, $u=1,2$. We remark that $k_{1}(\sigma)\neq k_{2}(\sigma)$ for all $\sigma$ since $I_{1}(\sigma)\neq I_{2}(\sigma)$.
 
 Let $\Pi^{\#}$ be the Langlands sum of $\Pi$, $\chi_{1}^{c}$ and $\chi_{2}^{c}$. We write the infinity type of $\Pi^{\#}$ at $\sigma\in\Sigma$ by $(z^{b_{i}(\sigma)}\overline{z}^{-b_{i}(\sigma)})_{1\leq i\leq n+2}$ with $b_{1}(\sigma)>b_{2}(\sigma)>\cdots >b_{n+2}(\sigma)$. The set $\{b_{i}(\sigma),1\leq i\leq n+2\}=\{a_{i}(\sigma),1\leq i\leq n\}\cup \{-k_{1}(\sigma),-k_{2}(\sigma)\}$.
 
 Let $\Pi^{\diamondsuit}$ be a cuspidalconjugate self-dual cohomological representation of $GL_{n+3}(\AF)$ with infinity type $(z^{c_{i}(\sigma)}\overline{z}^{-c_{i}(\sigma)})_{1\leq i\leq n+3}$ such that $-c_{n+3}(\sigma)>b_{1}(\sigma)>-c_{n+2}(\sigma)>b_{2}(\sigma)>\cdots >-c_{2}(\sigma)>b_{n+2}(\sigma)>-c_{1}(\sigma)$ for all $\sigma\in\Sigma$. We may assume that $\Pi^{\diamondsuit}$ has definable arithmetic automorphic periods.
 
Proposition \ref{Whittaker period theorem CM} is true for $(\Pi^{\diamondsuit},\Pi^{\#})$. Namely,
\begin{equation}
L(\cfrac{1}{2}+m,\Pi^{\diamondsuit}\times \Pi^{\#})\sim_{E(\Pi^{\diamondsuit})E(\Pi^{\#})} p(\Pi^{\diamondsuit})p(\Pi^{\#})p(m,\Pi^{\diamondsuit}_{\infty},\Pi^{\#}_{\infty}).
\end{equation}

We know \begin{equation}
L(\cfrac{1}{2}+m,\Pi^{\diamondsuit}\times \Pi^{\#})=L(\cfrac{1}{2}+m,\Pi^{\diamondsuit}\times \Pi)L(\cfrac{1}{2}+m,\Pi^{\diamondsuit}\times \chi_{1}^{c})L(\cfrac{1}{2}+m,\Pi^{\diamondsuit}\times \chi_{2}^{c})
\end{equation}

For $u=1$ or $2$, by Theorem \ref{n*1} and the fact that $\chi_{u}$ is conjugate self-dual, we have
\begin{eqnarray}
&L(\cfrac{1}{2}+m,\Pi^{\diamondsuit}\times \chi_{u}) &\nonumber\\
&\sim_{E(\Pi^{\diamondsuit})E(\Pi^{\#})} (2\pi i)^{(\frac{1}{2}+m)d(n+3)} P^{I(\Pi^{\diamondsuit},\chi_{u}^{c})}(\Pi)\prod\limits_{\sigma\in\Sigma} p(\widecheck{\chi_{u}},\sigma)^{-2I(\Pi^{\diamondsuit},\chi_{u}^{c})(\sigma)+(n+3)}.&\nonumber
\end{eqnarray}

Thoerem {Whittaker period theorem CM} implies that
\begin{equation}
p(\Pi^{\#})\sim_{E(\Pi^{\#})} \Omega(\Pi^{\#}_{\infty})p(\Pi)L(1,\Pi\otimes \chi_{1})L(1,\Pi\otimes \chi_{2})L(1,\chi_{1}\chi_{2}^{c})
\end{equation} where $\Omega(\Pi^{\#}_{\infty})$ is a non zero complex numbers depend on $\Pi^{\#}_{\infty}$.
 
 By Theorem \ref{n*1} again, for $u=1,2$, we have \begin{equation}
L(1,\Pi\times \chi_{u}) \sim_{E(\Pi^{\#})} (2\pi i)^{dn} P^{I(\Pi,\chi_{u})}\prod\limits_{\sigma\in\Sigma} p(\widecheck{\chi_{u}},\sigma)^{2I(\Pi,\chi_{u})(\sigma)-n}.
\end{equation}
 
 Moreover, $L(1,\chi_{1}\chi_{2}^{c})\sim_{E(\Pi^{\#})} (2\pi i)^{d}\prod\limits_{\sigma\in\Sigma}p(\widecheck{\chi_{1}},\sigma)^{t(\sigma)}p(\widecheck{\chi_{2}},\sigma)^{-t(\sigma)}$ where $t(\sigma)=1$ if $k_{1}(\sigma)<k_{2}(\sigma)$, $t(\sigma)=-1$ if $k_{1}(\sigma)>k_{2}(\sigma)$.

\begin{lem}
For all $\sigma\in\Sigma$, \begin{eqnarray}\nonumber
-2I(\Pi^{\diamondsuit},\chi_{1}^{c})(\sigma)+(n+3)=2I(\Pi,\chi_{1})(\sigma)-n+t(\sigma),\\\nonumber
-2I(\Pi^{\diamondsuit},\chi_{2}^{c})(\sigma)+(n+3)=2I(\Pi,\chi_{1})(\sigma)-n-t(\sigma).
\end{eqnarray}
\end{lem}
 \begin{proof}[Proof of the lemma:]
 
 By definition we have 
 \begin{equation}\nonumber I(\Pi^{\diamondsuit},\chi_{1}^{c})(\sigma)=\#\{1\leq i\leq n+3\mid -c_{i}(\sigma)>-k_{1}(\sigma)\}.
 \end{equation}

Recall that $-c_{n+3}(\sigma)>b_{1}(\sigma)>-c_{n+2}(\sigma)>b_{2}(\sigma)>\cdots >-c_{2}(\sigma)>b_{n+2}(\sigma)>-c_{1}(\sigma)$ and $\{b_{i}(\sigma),1\leq i\leq n+2\}=\{a_{i}(\sigma),1\leq i\leq n\}\cup \{-k_{1}(\sigma),-k_{2}(\sigma)\}$.

Therefore \begin{eqnarray}\nonumber
I(\Pi^{\diamondsuit},\chi_{1}^{c})(\sigma)&=&\#\{1\leq i\leq n+2\mid b_{i}(\sigma)>-k_{1}(\sigma)\}+1\\ \nonumber &=&
\#\{1\leq i\leq n\mid a_{i}(\sigma)>-k_{1}(\sigma)\}+\mathds{1}_{-k_{2}(\sigma)>-k_{1}(\sigma)}+1.
\end{eqnarray}

By definition we have
\begin{equation}\nonumber
I(\Pi,\chi_{1})(\sigma)=\#\{1\leq i\leq n\mid -a_{i}(\sigma)>k_{1}(\sigma)\}=n-\#\{1\leq i\leq n\mid a_{i}(\sigma)>-k_{1}(\sigma)\}.
\end{equation}

Therefore, $I(\Pi^{\diamondsuit},\chi_{1}^{c})(\sigma)=n-I(\Pi,\chi_{1})(\sigma)+\mathds{1}_{-k_{2}(\sigma)>-k_{1}(\sigma)}+1$. Hence we have
$-2I(\Pi^{\diamondsuit},\chi_{1}^{c})(\sigma))+(n+3)=2I(\Pi,\chi_{1})(\sigma)-n+1-2\mathds{1}_{-k_{2}(\sigma)>-k_{1}(\sigma)}$.

It is easy to verify that $1-2\mathds{1}_{-k_{2}(\sigma)>-k_{1}(\sigma)}=t(\sigma)$. The first statement then follows and the second is similar to the first one.

 \end{proof}

We deduce that if $L(\cfrac{1}{2}+m,\Pi^{\diamondsuit}\times \Pi^{\#})\neq 0$, then \begin{eqnarray}\nonumber\label{factorization complete final step}
&L(\cfrac{1}{2}+m,\Pi^{\diamondsuit}\times \Pi) (2\pi i)^{(1+2m)d(n+3)}P^{I(\Pi^{\diamondsuit},\chi_{1}^{c})}(\Pi^{\diamondsuit})P^{I(\Pi^{\diamondsuit},\chi_{2}^{c})}(\Pi^{\diamondsuit})&\\
&\sim_{E(\Pi^{\diamondsuit})E(\Pi^{\#})}(2\pi i)^{d(2n+1)}p(\Pi^{\diamondsuit})\Omega(\Pi^{\#}_{\infty})p(m,\Pi_{\infty}^{\#},\Pi^{\#}_{\infty})P^{I(\Pi,\chi_{1})}(\Pi)P^{I(\Pi,\chi_{2})}(\Pi).& \nonumber
\end{eqnarray}

Now let $\chi_{1}'$, $\chi_{2}'$ be two conjugate self-dual algebraic Hecke characters of $F$ such that $\chi'_{1,\sigma}=\chi_{1,\sigma}$ and $\chi'_{2,\sigma}=\chi_{2,\sigma}$ for $\sigma\neq \sigma_{0}$, $\chi'_{1,\sigma_{0}}=\chi_{2,\sigma_{0}}$ and $\chi'_{2,\sigma_{0}}=\chi_{1,\sigma_{0}}$.

We take $\Pi^{\#\#}$ as Langlands sum of $\Pi$, $\chi'_{1}\text{}^{c}$ and $\chi'_{2}\text{}^{c}$. Since the infinity type of $\Pi^{\#\#}$ is the same with $\Pi^{\#}$, we can repeat the above process and we see that equation (\ref{factorization complete final step}) is true for $(\Pi^{\diamondsuit},\Pi^{\#\#})$. Observe that most terms remain unchanged. \\

Comparing equation (\ref{factorization complete final step}) for $(\Pi^{\diamondsuit},\Pi^{\#})$ and that for $(\Pi^{\diamondsuit},\Pi^{\#\#})$, we get

\begin{equation}\label{last compare factorization}
\cfrac{P^{I(\Pi^{\diamondsuit},\chi'_{1}\text{}^{c})}(\Pi^{\diamondsuit})P^{I(\Pi^{\diamondsuit},\chi'_{2}\text{}^{c})}(\Pi^{\diamondsuit})}{P^{I(\Pi^{\diamondsuit},\chi_{1}^{c})}(\Pi^{\diamondsuit})P^{I(\Pi^{\diamondsuit},\chi_{2}^{c})}(\Pi^{\diamondsuit})} \sim_{E(\Pi^{\diamondsuit})E(\Pi)}
\cfrac{P^{I(\Pi,\chi'_{1})}(\Pi)P^{I(\Pi,\chi'_{2})}(\Pi)}{P^{I(\Pi,\chi_{1})}(\Pi)P^{I(\Pi,\chi_{2})}(\Pi)}.
\end{equation}

By construction, $I(\Pi,\chi_{u})=I_{u}$ and $I(\Pi,\chi'_{u})=I'_{u}$ for $u=1,2$. Hence to prove (\ref{factorization equation}), it is enough to show the left hand side of the above equation is a number in $E(\Pi^{\diamondsuit})^{\times}$.

There are at least two ways to see this. We observe that $I(\Pi^{\diamondsuit},\chi'_{1}\text{}^{c})(\sigma)=I(\Pi^{\diamondsuit},\chi_{1}\text{}^{c})(\sigma)$, $I(\Pi^{\diamondsuit},\chi'_{2}\text{}^{c})(\sigma)=I(\Pi^{\diamondsuit},\chi_{2}\text{}^{c})(\sigma)$ for $\sigma\neq \sigma_{0}$ and $I(\Pi^{\diamondsuit},\chi'_{1}\text{}^{c})(\sigma_{0})=I(\Pi^{\diamondsuit},\chi_{2}\text{}^{c})(\sigma_{0})$, $I(\Pi^{\diamondsuit},\chi'_{2}\text{}^{c})(\sigma_{0})=I(\Pi^{\diamondsuit},\chi_{1}\text{}^{c})(\sigma_{0})$. Moreover, these numbers are all in $\{1,2,\cdots,(n+3)-1\}$. Theorem \ref{restricted theorem} gives a factorization of the holomorphic arithmetic automorphic periods through each place. In particular, it implies that the left hand side of  (\ref{last compare factorization}) is in $E(\Pi^{\diamondsuit})^{\times}$ as expected.

One can also show this by taking $\Pi^{\diamondsuit}$ an automorphic induction of a Hecke character. We can then calculate $L(\cfrac{1}{2}+m,\Pi^{\diamondsuit}\times\chi_{u}^{c})$ in terms of CM periods. Since the factorization of CM periods is clear, we will also get the expected result.

When $n$ is even, we consider $\Pi^{\#}$ the Langlands sum of $\Pi$, $(\chi_{1}\psi||\cdot||^{-1/2})^{c}$ and $(\chi_{2}\psi||\cdot||^{-1/2})^{c}$ where $\chi_{1}$, $\chi_{2}$ are two suitable algebraic Hecke characters of $F$. We follow the above steps and will get the factorization in this case. We leave the details to the reader.

\end{proof}

\subsection{Specify the factorization}\label{specify the factorization}
\text{}
 Let us assume that Conjecture \ref{factorization conjecture} is true. We want to specify one factorization. \\
 
 We denote by $I_{0}$ the map which sends each $\sigma\in\Sigma$ to $0$. By the last part of Corollary \ref{factorization lemma}, it is enough to choose $c(\Pi,\sigma)\in (\C/E(\Pi))^{\times}$ which is $G_{K}$-equivariant such that $P^{(I_{0})}(\Pi)\sim_{E(\Pi)} \prod\limits_{\sigma\in\Sigma}c(\Pi,\sigma)$. Then there exists a unique factorization of $P^{(\cdot)}(\Pi)$ such that $P^{(0)}(\Pi,\sigma)= c(\Pi,\sigma)$ . We may then define the \textbf{local arithmetic automorphic periods} $P^{(s)}(\Pi,\sigma)$ as an element in $\C^{\times}/ (E(\pi))^{\times}$.
 
In this section, we shall prove $P^{(I_{0})}(\Pi)\sim_{E(\Pi)} p(\widecheck{\xi_{\Pi}},\overline{\Sigma}) \sim_{E(\Pi)} \prod\limits_{\sigma\in\Sigma} p(\widecheck{\xi_{\Pi}},\overline{\sigma})$. Therefore, we may take $c(\Pi,\sigma)=p(\widecheck{\xi_{\Pi}},\overline{\sigma})$.

More generally, we will see that:
\begin{lem}\label{period in compact case}
If $I$ is compact then $P^{(I)}(\Pi)\sim_{E(\Pi)}  \prod\limits_{I(\sigma)=0} p(\widecheck{\xi_{\Pi}},\overline{\sigma})\times \prod\limits_{I(\sigma)=n} p(\widecheck{\xi_{\Pi}},\sigma)$. 
\end{lem}

This lemma leads to the following theorem:
\begin{thm}\label{special factorization theorem}
If Conjecture \ref{factorization conjecture} is true, in particular, if conditions in Theorem \ref{complete theorem} are satisfied, then there exists some complex numbers $P^{(s)}(\Pi,\sigma)$ unique up to multiplication by elements in $(E(\Pi))^{\times}$ such that the following two conditions are satisfied:
\begin{enumerate}
\item $P^{(I)}(\Pi) \sim_{E(\Pi)} \prod\limits_{\sigma\in\Sigma}P^{(I(\sigma))}(\Pi,\sigma)$ for all $I=(I(\sigma))_{\sigma\in\Sigma}\in \{0,1,\cdots,n\}^{\Sigma}$,
\item  and $P^{(0)}(\Pi,\sigma)\sim_{E(\Pi)} p(\widecheck{\xi_{\Pi}},\overline{\sigma})$
\end{enumerate}
where $\xi_{\Pi}$ is the central character of $\Pi$.

Moreover, we know $P^{(n)}(\Pi,\sigma)\sim_{E(\Pi)} p(\widecheck{\xi_{\Pi}},\sigma)$ or equivalently $P^{(0)}(\Pi,\sigma)\times P^{(n)}(\Pi,\sigma)\sim_{E(\Pi)} 1$.
\end{thm}

\paragraph{Proof of Lemma \ref{period in compact case}:}
Recall that $D/2=\sum\limits_{\sigma\in\Sigma}I_{\sigma}(n-I_{\sigma})=0$ since $I$ is compact.

Let $T$ be the center of $GU_{I}$. We have $$T(\R)\cong \{(z_{\sigma})\in (\C^{\times})^{\Sigma}\mid |z_{\sigma}| \text{ does not depend on }\sigma\}.$$
We define a homomorphism $h_{T}:\Ss(\R) \rightarrow T(\R)$ by sending $z\in \C$ to $((z)_{I(\sigma)=0},(\overline{z})_{I(\sigma)=n})$. \\

Since $I$ is compact, we see that $h_{I}$ is the composition of $h_{T}$ and the embedding $T\hookrightarrow GU_{I}$. We get an inclusion of Shimura varieties:
$Sh_{T}:=Sh(T,h_{T})\hookrightarrow Sh_{I}=Sh(GU_{I},h_{I})$.

Let $\xi$ be a Hecke character of $K$ such that $\Pi^{\vee}\otimes \xi$ descends to $\pi$, a representation of $GU_{I}(\AQ)$, as before. We write $\lambda\in \Lambda(GU_{I})$ the cohomology type of $\pi$. We define $\lambda^{T}:=(\lambda_{0},(\sum\limits_{1\leq i\leq n}\lambda_{i}(\sigma))_{\sigma\in\Sigma})$. Since $\pi$ is irreducible, it acts as scalars when restrict to $T$. This gives $\pi^{T}$, a one dimensional representation of $T(\AQ)$ which is cohomology of type ${\lambda}^{T}$. We denote by $V_{\lambda^{T}}$ the character of $T(\R)$ with highest weight $\lambda^{T}$. 

The automorphic vector bundle $E_{\lambda}$ pulls back to the automorphic vector bundle $[V_{\lambda^{T}}]$ (see \cite{harrisappendix} for notation) on $Sh_{T}$. \\

Let $\beta$ be an element in $\bar{H}^{0}(Sh_{I},E_{\lambda})^{\pi}$. We fix a non zero $E(\pi)$-rational element in $\pi$ and then we can lift $\beta$ to $\phi$, an automorphic form on $GU_{I}(\AQ)$.

There is an isomorphism $H^{0}(Sh_{T},[V_{\lambda^{T}}])\xrightarrow{\sim} \{f\in \mathbb{C}^{\infty}(T(\Q)\backslash T(\AQ),\C\mid f(tt_{\infty}))=\pi^{T}(t_{\infty})f(t), t_{\infty}\in T(\R), t\in T(\AQ)\}$ (c.f. \cite{harrisappendix}). We send $\beta$ to the element in $H^{0}(Sh_{T},[V_{\lambda^{T}}])^{\pi^{T}}$ associated to $\phi|_{T(\AQ)}$.

We then obtain rational morphisms
\begin{eqnarray}
&\bar{H}^{0}(Sh_{I},E_{\lambda})^{\pi} \xrightarrow{\sim} H^{0}(Sh_{T},[V_{\lambda^{T}}])^{\pi^{T}}&\\
\text{ and similarly } &\bar{H}^{0}(Sh_{I},E_{\lambda^{\vee}})^{\pi^{\vee}} \xrightarrow{\sim} H^{0}(Sh_{T},[V_{\lambda^{T,\vee}}])^{\pi^{T,\vee}}.&
\end{eqnarray}

These morphisms are moreover isomorphisms. In fact, since both sides are one dimensional, it is enough to show the above morphisms are injective. Indeed, if $\phi$, a lifting of an element in $\bar{H}^{0}(Sh_{I},E_{\lambda})^{\pi}$, vanishes at the center, in particular, it vanishes at the identity. Hence it vanishes at $GU_{I}(\AQf)$ since it is an automorphic form. We observe that $GU_{I}(\AQf)$ is dense in $GU_{I}(\Q)\backslash GU_{I}(\AQ)$. We know $\phi=0$ as expected.

We are going to calculate the arithmetic automorphic period. Let $\beta$ be rational. We take a rational element $\beta^{\vee}\in \bar{H}^{0}(Sh_{I},E_{\lambda^{\vee}})^{\pi^{\vee}} $ and lift it to an automorphic form $\phi^{\vee}$. We have $ c_{B}(\phi) \sim_{E(\pi)}P^{(I)}(\pi) \phi^{\vee} $ by Lemma \ref{pair to ratio}.

For the torus, by Remark \ref{CM complex conjugation}, we know \begin{equation}\nonumber
\phi^{\vee}|_{T(\AQ)}\sim_{E(\pi)} p(Sh(T,h_{T}),\pi^{T})^{-1}(\phi|_{T(\AQ)})^{-1}.
\end{equation}

Recall that $c_{B}(\phi)=\pm i^{\lambda_{0}}\overline{\phi}||\nu(\cdot)||^{\lambda_{0}}$. Therefore $(c_{B}(\phi))|_{T(\AQ)}=\pm i^{\lambda_{0}}(\phi|_{T(\AQ)})^{-1}$. We then get \begin{equation}\label{pass to torus}
 i^{\lambda_{0}}P^{(I)}(\pi)\sim_{E(\pi)}p(Sh(T,h_{T}),\pi^{T}).
\end{equation}

We now set $T^{\#}:=Res_{K/\Q}T_{K}$. We have $T^{\#} \cong Res_{K/\Q}\mathbb{G}_{m}\times Res_{F/\Q}\mathbb{G}_{m}$. In particular, $T^{\#}(\R)\cong \C^{\times} \times (\R\otimes_{\Q}F)^{\times}\cong \C^{\times}\times (\C^{\times})^{\Sigma}$.

We define $h_{T^{\#}}:\Ss(\R) \rightarrow T^{\#}(\R)$ to be the composition of $h_{T}$ and the natural embedding $T(\R)\rightarrow T^{\#}(\R)$. We know $h_{T^{\#}}$ sends $z\in\C^{\times}$ to $(z\overline{z},(z)_{I(\sigma)=0}, (\overline{z})_{r(\sigma)=0})$. The embedding $(T,h_{T})\rightarrow (T^{\#},h_{T^{\#}})$ is a map between Shimura datum.

We observe that $\pi^{T,\#}:=||\cdot||^{-\lambda_{0}}\times \xi_{\Pi}^{-1}$ is a Hecke character on $T^{\#}$. Its restriction to $T$ is just $\pi^{T}$. By Proposition \ref{propgeneral}, we have $p(Sh(T,h_{T}),\pi^{T})\sim_{E(\pi)}p(Sh(T^{\#},h_{T^{\#}}),\pi^{T^{\#}})$. \\

By the definition of CM period and Proposition \ref{propCM}, we have
 \begin{equation}
p(Sh(T^{\#},h_{T^{\#}}),\pi^{T^{\#}}) \sim_{E(\pi)} (2\pi i)^{\lambda_{0}} \prod\limits_{I(\sigma)=0}p(\xi_{\Pi}^{-1},\sigma) \prod\limits_{I(\sigma)=n}p(\xi_{\Pi}^{-1},\overline{\sigma}).
\end{equation}

Since $\xi_{\Pi}$ is conjugate self-dual, we have $p(\xi_{\Pi}^{-1},\overline{\sigma})\sim_{E(\Pi)} p(\xi_{\Pi},\sigma)$.

By equation (\ref{pass to torus}), we get:
\begin{equation}
 i^{\lambda_{0}}P^{(I)}(\pi)\sim_{E(\pi)}(2\pi i)^{\lambda_{0}} \prod\limits_{I(\sigma)=0}p(\xi_{\Pi}^{-1},\sigma) \prod\limits_{I(\sigma)=n}p(\xi_{\Pi},\sigma).
\end{equation}
Recall that by definition $P^{(I)}(\Pi)\sim_{E(\Pi)} (2\pi)^{-\lambda_{0}} P^{(I)}(\pi)$,  we get finally
\begin{eqnarray}\nonumber
P^{(I)}(\Pi)&\sim_{E(\Pi)}&  \prod\limits_{I(\sigma)=0} p(\xi_{\Pi}^{-1},\sigma)\times \prod\limits_{I(\sigma)=n} p(\xi_{\Pi},\sigma)\\\nonumber
&\sim_{E(\Pi)}&  \prod\limits_{I(\sigma)=0} p(\widecheck{\xi_{\Pi}},\overline{\sigma})\times \prod\limits_{I(\sigma)=n} p(\widecheck{\xi_{\Pi}},\sigma). 
\end{eqnarray} 
The last formula comes from the fact that $\xi_{\Pi}$ is conjugate self-dual.

\begin{flushright}$\Box$\end{flushright}

\begin{rem}\label{remark n=1}
If $n=1$ and $\Pi=\eta$ is a Hecke character, we obtain that:
$
P^{(0)}(\eta,\sigma)\sim_{E(\eta)} p(\widecheck{\eta},\overline{\sigma})
$
and similarly
$
P^{(1)}(\eta,\sigma)\sim_{E(\eta)}p(\widecheck{\eta},\sigma).$
\end{rem}

\footnotesize
{\sc Jie Lin: Institut des hautes \'{e}tudes scientifiques, 35 Route de Chartres, 91440 Bures-sur-Yvette, France.}
\\ {\it E-mail address:} {\tt linjie@ihes.fr}

\bibliography{bibfile}
\bibliographystyle{alpha}

\end{document}